\documentclass[11pt,reqno]{amsart}
\usepackage{graphicx}

\usepackage{amssymb}
\usepackage{cite}
\usepackage{amsmath}
\usepackage{latexsym}
\usepackage{amscd}
\usepackage{amsthm}
\usepackage{mathrsfs}
\usepackage{url}

\usepackage{color}

\usepackage[backref=page]{hyperref}
\numberwithin{equation}{section}

\newtheorem{thm}{Theorem}

\newtheorem{prop}[thm]{Proposition}

\theoremstyle{definition}

\theoremstyle{remark}
\newtheorem{rem}[thm]{Remark}

\def\R{\mathbb R}

\def\N{\mathbb N}

\def\SS{\mathbb S}

\def\pt{\partial}

\allowdisplaybreaks

\begin{document}
\title[]{On shape optimization for fourth order Steklov eigenvalue problems}
\author[C.~Xiong]{Changwei~Xiong}
\address{School of Mathematics, Sichuan University, Chengdu 610065, Sichuan,  P.~R.~China}
\email{\href{mailto:changwei.xiong@scu.edu.cn}{changwei.xiong@scu.edu.cn}}
\author[J.~Yang]{Jinglong~Yang}
\address{School of Physics, Huazhong University of Science and Technology, Wuhan 430074, Hubei,  P.~R.~China}
\email{\href{mailto:yangjinglong@hust.edu.cn}{yangjinglong@hust.edu.cn}}
\author[J.~Yu]{Jinchao~Yu}
\address{School of Mathematics, Sichuan University, Chengdu 610065, Sichuan,  P.~R.~China}
\email{\href{mailto:jcyu97@126.com}{jcyu97@126.com}}
\date{\today}
\thanks{This research was supported by NSFC (Grant no. 12171334) and National Key R and D Program of China 2021YFA1001800.}

\maketitle

\begin{abstract}
We study three types of fourth-order Steklov eigenvalue problems. For the first two of them, we derive the asymptotic expansion of their spectra on Euclidean annular domains $\mathbb{B}^n_1\setminus \overline{\mathbb{B}^n_\epsilon}$ as $\epsilon \to 0$, leading to conclusions on shape optimization. For these two problems, we also compute their spectra on cylinders over closed Riemannian manifolds. Last, for the third problem, we obtain a sharp upper bound for its first non-zero eigenvalue on star-shaped and mean convex Euclidean domains.
\end{abstract}

\section{Introduction}\label{sec1}

The second-order Steklov eigenvalue problem is a classical problem in the differential geometry and the partial differential equations, first introduced by Steklov \cite{SW02} around 1900. See \cite{CGGS24,GP17,KKK14} for an overview on this problem. Different from the well-known Laplace eigenvalue problem, the parameter (eigenvalue) of the Steklov eigenvalue problem appears on the boundary. Regarding this eigenvalue problem, the estimate on its eigenvalues is one of the main research topics in recent years. In particular, people are concerned with the shape optimization of Steklov eigenvalue problems. For instance, for the sharp upper bound of the first non-zero Steklov eigenvalue,  there are the classical Weinstock--Brock type results \cite{Wei54,Bro01} for Euclidean domains. Recently, an upper bound in the same spirit for convex Euclidean domains was given by Bucur, Ferone, Nitsch, and Trombetti \cite{BFNT}. Later, Kwong and Wei \cite{KW23} were able to extend the result in \cite{BFNT} to star-shaped and mean convex Euclidean domains. On the other hand, without the convexity condition, Fraser and Schoen \cite{FS19} in 2019 concluded that among all smooth contractible Euclidean domains of the dimension $n\geq 3$ with a fixed boundary area, the ball is not the one that maximizes the first non-zero Steklov eigenvalue. This makes the shape optimization for contractible Euclidean domains of the dimension $n\geq 3$ more complicated.

In this paper, we study some shape optimization problems for three types of fourth-order Steklov eigenvalue problems. The first two of them, one with the Neumann boundary condition and the other with the Dirichlet boundary condition, are classical and quite well-known. They were first introduced and studied by Kuttler and Sigillito \cite{KS68} in 1968 (see also Payne's work \cite{P70} in 1970), have important applications in biharmonic analysis and elasticity, and have received continuing attention from mathematicians; see \cite{Xio21,Xio22} and references therein for an introduction. In contrast, the third one was studied more recently, introduced first by Buoso and Provenzano \cite{BP15} in 2015. We shall study these three problems case by case.

Let $(\Omega^n, g)$ $(n \geq 2)$ be an $n$-dimensional connected compact smooth Riemannian manifold with a smooth boundary $\partial \Omega$. In the first part of the paper we consider the fourth-order Steklov eigenvalue problem of the type one:
\begin{equation}\label{problem1}
	\begin{cases}
		\Delta ^{2}u =0,&\text{ in }  \Omega, \\
		\dfrac{\partial u }{\partial \nu }=0, \ \dfrac{\partial \left ( \Delta u  \right )}{\partial \nu }+\xi u =0,&\text{ on }  \partial \Omega,
	\end{cases}
\end{equation}
where $\nu$ denotes the outward unit normal to $\partial \Omega$. The spectrum of the problem $\eqref{problem1}$ consists of a discrete and increasing sequence of eigenvalues counted with the multiplicity
\begin{align*}
	0=\xi_{0}(\Omega)<\xi_{1}(\Omega) \leq \xi_{2}(\Omega) \leq \cdots \nearrow +\infty.
\end{align*}
In contrast, we use $\xi_{(k)}(\Omega)$ to denote the eigenvalues without the multiplicity. (Similar notations apply to other eigenvalues in this paper.) And the eigenvalue $\xi_{k}$ admits the variational characterization:
\begin{align*}
	\xi_{k}=\inf_{u \in H^{2}(\Omega), \pt_\nu u|_{\pt \Omega}=0, u|_{\pt \Omega}\neq 0, \atop \int_{\partial \Omega} u u_{i} d a_{g}, i=0,1,\cdots,k-1}
	\dfrac{\int_{\Omega} \left(\Delta u\right)^{2}  d v_{g}}{\int_{\partial \Omega}  u^{2} d a_{g}},
\end{align*}
where $u_{i}$ is the $i$th eigenfunction. For the Euclidean unit ball $\Omega=\mathbb{B}^n_1$ ($n \geq 2$), the eigenvalues of $\eqref{problem1}$ are $k^{2} (n+2k)$ ($k\geq 0$), which was derived by Xia and Wang \cite{XW18} (cf. Section~5.2 in Liu's \cite{Liu16}).

In addition, for the eigenvalue problem $\eqref{problem1}$, Xia and Wang \cite{XW18} found an isoperimetric upper bound for the first non-zero eigenvalue $\xi_{1}$ on a bounded domain $\Omega$ with a smooth boundary in $\mathbb{R}^{n}$, which states
\begin{align*}
	\xi_{1}(\Omega) \le \frac{(n+2) |\pt \Omega|  }{n |\Omega| (|\Omega| /|\mathbb{B}^n_1| )^{2/n}}
\end{align*}
with the equality if and only if $\Omega$ is a ball. We are interested in other kinds of isoperimetric inequalities. First note that we have the rescaling property $\xi_k((\Omega,c^2g))=c^{-3}\xi_k((\Omega,g))$ ($c>0$). So in order to study the shape optimization problem, we need to do some normalization on the eigenvalues. The most natural one for the Steklov problem~\eqref{problem1} is possibly $\xi_k(\Omega)|\pt \Omega|^{3/(n-1)}$. Nevertheless, we may also study $\xi_k(\Omega)| \Omega|^{3/n}$, or more generally,
\begin{align}\label{n-eigenvalue1}
\xi_k(\Omega)(|\pt \Omega|^{3/(n-1)})^\alpha (| \Omega|^{3/n})^{1-\alpha}
\end{align}
for a parameter $\alpha\in \R$. Then for the purpose of establishing an isoperimetric inequality involving the quantity \eqref{n-eigenvalue1}, it is advisable to examine it first for simple examples.

Except the Euclidean balls, the Euclidean annular domains may be the simplest examples for which we can verify an isoperimetric inequality. Thus we first consider the asymptotic expansion of the spectrum of the problem \eqref{problem1} on the Euclidean annular domain. Below is our result in this aspect.
\begin{thm}\label{thm1}
Let $n\geq 2$ and $k\geq 1$. For the problem \eqref{problem1} on the Euclidean annular domain $\mathbb {B}^{n}_{1} \setminus \overline{\mathbb {B}^{n}_{\epsilon}}$, its $k$th distinct Steklov eigenvalue $\xi_{(k)}$ has the following asymptotic expansion as $\epsilon \to 0$.
	\begin{itemize}
		\item[{\rm (i)}]  When $n+2k \geq 7$,
		\begin{equation*}
			\xi_{(k)} = k^2(n+2k)
			-k^2 (n+2k)\left[\frac{k(n+2k)}{n+k-2}+2\right]\epsilon^{n+2k-2}
			+O\left(\epsilon^{n+2k-1}\right).
		\end{equation*}
		\item[{\rm (ii)}]  When $n+2k=6$,
	  \begin{align*}
	  	\xi_{(k)}
	  	=&k^{2}(n+2k)
	  	-k^{2}(n+2k)\left[\frac{k(n+2k)}{n+k-2}+\frac{1}{2}(n+2k-3)^{3}\right.\\
	  	&\left.~+\frac{1}{2}(n+2k+1)\right]\epsilon^{n+2k-2} +O\left(\epsilon^{n+2k-1}\right).
	  \end{align*}
	\item[{\rm (iii)}]  When $n+2k=5$, namely, $n=3$, $k=1$,
	\begin{align*}
		\xi_{(1)}=5-\dfrac{1485}{64}\epsilon^{3}+O\left(\epsilon^{4}\right).
	\end{align*}
\item[{\rm (iv)}] When  $n=2,~k=1$,
  \begin{align*}
	\xi_{(1)}
	=4-24\epsilon^{2}+O\left(\epsilon^{2} \log^{-1}\epsilon\right).
  \end{align*}
	\end{itemize}
In particular, the first non-zero eigenvalue as $\epsilon \to 0$ is given by
\begin{align*}
	\xi_{(1)}=
\begin{cases}
4-24\epsilon^{2}+O\left(\epsilon^{2} \log^{-1}\epsilon\right),\quad n=2,\\
5-\dfrac{1485}{64}\epsilon^{3}+O\left(\epsilon^{4}\right),\quad  n=3,\\
6-114\epsilon^{4}+O\left(\epsilon^{5}\right),\quad n=4,\\
n + 2 - \dfrac{3n(n+2)}{n-1}\epsilon^{n} + O\left(\epsilon^{n+1}\right),\quad n\geq5,
\end{cases}
\end{align*}
and $\xi_{(k)}\left(\mathbb {B}^{n}_{1} \setminus \overline{\mathbb {B}^{n}_{\epsilon}}\right)$ satisfies
\begin{align*}
	\xi_{(k)}\left(\mathbb {B}^{n}_{1} \setminus \overline{\mathbb {B}^{n}_{\epsilon}}\right)
	\left|  \partial \left( \mathbb {B}^{n}_{1}  \setminus \overline{\mathbb {B}^{n}_{\epsilon}} \right)\right| ^{\frac{3}{n-1}}
	>
	\xi_{(k)}\left(\mathbb {B}^{n}_{1}\right)
	\left|  \partial \left( \mathbb {B}^{n}_{1}\right)\right| ^{\frac{3}{n-1}}.
\end{align*}
\end{thm}
As an immediate consequence, we get the following conclusion on the shape optimization, which is the main purpose of our study. Here we also would like to point out that since there are few results on the shape optimization of fourth-order Steklov eigenvalues, any results on this topic would be precious in some sense.
\begin{thm}\label{main-purpose1}
For the normalized eigenvalue $\xi_m(\Omega)|\pt \Omega|^{3/(n-1)}$ ($m\geq 1$), the Euclidean balls are not the maximizers among Euclidean domains.
\end{thm}
By Theorem~\ref{main-purpose1}, it remains an interesting question which domains are the maximizers of the normalized eigenvalue $\xi_m(\Omega)|\pt \Omega|^{3/(n-1)}$ ($m\geq 1$).
\begin{rem}
Similarly, we may draw conclusions for other normalized eigenvalues in \eqref{n-eigenvalue1}.
\end{rem}
\begin{rem}
The analogous result as in Theorem~\ref{thm1} for the second-order Steklov eigenvalue problem was obtained by Fraser and Schoen \cite{FS19} and others (see their paper \cite{FS19} for references). Moreover, they \cite{FS19} were able to carry out a one-dimensional surgery to draw an isoperimetric-type conclusion for contractible Euclidean domains of the dimension $n\geq 3$. See Hong's \cite{Hon21} for the higher dimensional surgery and related results. It is an interesting question whether similar surgeries may apply to fourth-order Steklov eigenvalue problems in the present paper.
\end{rem}
As for the proof of Theorem~\ref{thm1} (and also of Theorem~\ref{thm3} below), we use the method of separating variables. However, the computations involved are lengthy somewhere. Although all the computations may be carried out by hand, we recommend using computer softwares like Mathematica to assist with. We also remark that though straightforward, the computations are not easy.

Next we turn to another important example for Steklov eigenvalue problems, the cylinder over a closed Riemannian manifold. See Example~2.2 in \cite{CGGS24} for the second-order problem. Our next result is the explicit computation of the fourth-order Steklov spectrum for this kind of manifolds by the method of separating variables, and it may be used in the future for the examination of certain isoperimetric inequality or for the study of the spectral properties (cf. Liu's works \cite{Liu11, Liu16}).
\begin{thm}\label{thm2}
Let $C_{L}=[-L,L] \times M$ be a cylinder for $L \in \mathbb{R^{+}}$, where $M$ is a compact Riemannian manifold without boundary. Denote the spectrum of the Laplace--Beltrami operator $\Delta_{M}$ on $M$ by
	\begin{align*}
		0=\lambda_{0}(M)<\lambda_{1} (M) \leq \lambda_{2} (M) \leq \cdots,
	\end{align*}
	and let $\{\beta_k\}^{\infty}_{k=0} \subset L^{2}(M)$ be a corresponding orthonormal basis of eigenfunctions such that
	\begin{equation}\label{DSigma1}
		-\Delta_{M} \beta_k=\lambda_{k}(M) \beta_k.
	\end{equation}
Let $c_k=\lambda_k(M)^{1/2}$. Then the eigenvalues of the cylinder $C_{L}$ for the fourth-order Steklov eigenvalue problem $\eqref{problem1}$ are given by
	\begin{equation*}
		0, ~
		\frac{3}{L^3}, ~
		\frac{2c_{k}^{3} \left(e^{2Lc_{k}}+1\right)^2}{e^{4Lc_{k}}-4Lc_{k}e^{2Lc_{k}}-1}, ~
		\frac{2c_{k}^{3} \left(e^{2Lc_{k}}-1\right)^2}{e^{4Lc_{k}}+4Lc_{k}e^{2Lc_{k}}-1}
	\end{equation*}
	and the corresponding eigenfunctions $u(s,x)$ are
	\begin{align*}
	&~1, ~
	-3L^2 s+s^3, 	\\
	&\left(\frac{c_{k} L-1-P_{k}}{c_{k} e^{2Lc_{k}}+c_{k}} e^{c_{k}s}
	+s e^{c_{k}s}
	+\frac{1-c_{k} L+P_{k}}{c_{k} e^{2Lc_{k}}+c_{k}} e^{-c_{k}s}
	+s e^{-c_{k}s}\right) \beta_k (x),\\
	&\left(\frac{c_{k} L-1+P_{k}}{c_{k} e^{2Lc_{k}}-c_{k}} e^{c_{k}s}
	-s e^{c_{k}s}
	+\frac{c_{k} L-1+P_{k}}{c_{k} e^{2Lc_{k}}-c_{k}} e^{-c_{k}s}
	+s e^{-c_{k}s}\right) \beta_k (x),
   \end{align*}
where $P_{k}:=\left(1+c_{k} L\right)e^{2L c_{k}}$.
\end{thm}
\begin{rem}\label{remark1}
Note that $\Sigma:=\pt C_L$ comprises two copies of $M$. So we see $\lambda_{2k}(\Sigma)=\lambda_{2k+1}(\Sigma)=\lambda_k(M)$, $k=0,1,2,\dots$. In addition, we have the following observations.
\begin{itemize}
\item[{\rm (i)}]  Consider the asymptotics of the Steklov eigenvalues $\xi_m$ as $m \rightarrow \infty$. The eigenvalues satisfy
		\begin{align*}
			\xi_{m}(C_{L})=2\lambda_{m}(\Sigma)^{3/2}+O\left(m^{-\infty}\right),\text{ as } m\to \infty,
		\end{align*}
where $O\left(m^{-\infty}\right)$ as $m\to \infty$ means that $\lim_{m\to \infty}O(m^{-\infty})/m^{-l}$ is bounded for any $l\in \N$. This may be verified by use of the Weyl law $\lambda_k(M^{n-1})\sim c k^{2/(n-1)}$ as $k\to \infty$.
		\item[{\rm (ii)}]  Consider the limiting behaviour as $L \rightarrow 0$. The eigenvalues satisfy
		\begin{equation*}
			\begin{gathered}
				\frac{3}{L^3} \rightarrow \infty,~ \frac{2c_{k}^{3} \left(e^{2Lc_{k}}+1\right)^2}{e^{4Lc_{k}}-4Lc_{k}e^{2Lc_{k}}-1} \rightarrow \infty,~
				\frac{2c_{k}^{3} \left(e^{2Lc_{k}}-1\right)^2}{e^{4Lc_{k}}+4Lc_{k}e^{2Lc_{k}}-1} \rightarrow 0.
			\end{gathered}
		\end{equation*}
		\item[{\rm (iii)}]  Consider the limiting behaviour as $L \rightarrow \infty$. The eigenvalues satisfy
		\begin{equation*}
			\begin{gathered}
				\frac{3}{L^3} \rightarrow 0,~
				\frac{2c_{k}^{3} \left(e^{2Lc_{k}}+1\right)^2}{e^{4Lc_{k}}-4Lc_{k}e^{2Lc_{k}}-1} \rightarrow 2c_{k}^{3},~
				\frac{2c_{k}^{3} \left(e^{2Lc_{k}}-1\right)^2}{e^{4Lc_{k}}+4Lc_{k}e^{2Lc_{k}}-1} \rightarrow 2c_{k}^{3}.
			\end{gathered}
		\end{equation*}
So we see $\xi_m(C_L)\to 2\lambda_m(\Sigma)^{3/2}$ as $L\to \infty$.
	\end{itemize}
\end{rem}
By (ii) of Remark~\ref{remark1}, we get the following result.
\begin{thm}\label{thm-existence1}
 There exist Riemannian manifolds for which the $m$th Steklov eigenvalue $\xi_m$ ($m\geq 1$ fixed) is arbitrarily small while keeping the boundary area fixed.
\end{thm}
Therefore, in order to study lower bounds of $\xi_m(\Omega)|\pt \Omega|^{3/(n-1)}$ ($m\geq 1$) for a general Riemannian manifold, we have to impose certain geometric conditions.

In the second part of the paper, we consider the fourth-order Steklov eigenvalue problem of the type two:
\begin{equation}\label{problem2}
	\begin{cases}
		\Delta ^{2} u =0,&\text{ in }  \Omega, \\
		u =0, \ \Delta u =\eta \dfrac{\partial u }{\partial \nu },&\text{ on }  \partial \Omega,
	\end{cases}
\end{equation}
where $\nu$ denotes the outward unit normal to $\partial \Omega$. The eigenvalues of the problem $\eqref{problem2}$ comprise a discrete and increasing sequence counted with the multiplicity
\begin{align*}
 	0<\eta_{0}(\Omega)<\eta_{1}(\Omega) \leq \eta_{2}(\Omega) \leq \cdots \nearrow +\infty.
\end{align*}
And the eigenvalue $\eta_{k}$ admits the variational characterization:
\begin{align*}
	\eta_{k}=\inf_{u \in H^{2}(\Omega), u|_{\pt \Omega}=0, \pt_\nu u|_{\pt \Omega}\neq 0, \atop \int_{\partial \Omega} u u_{i} d a_{g}, i=0,1,\cdots,k-1}
	\dfrac{\int_{\Omega} \left(\Delta u \right)^{2}  d v_{g}}{\int_{\partial \Omega} \left(\partial_{\nu} u \right)^{2} d a_{g}},
\end{align*}
where $u_{i}$ is the $i$th eigenfunction.


The first eigenvalue $\eta_0$ is significant since as observed by Kuttler \cite{K72,K79}, it is the sharp constant for the $L^{2}$-a priori estimate of the Laplace equation with nonhomogeneous Dirichlet boundary conditions. If $\Omega$ is a Euclidean unit ball, then the eigenvalues of the problem \eqref{problem2} are $n+2k$ ($k\geq 0$) \cite{FGW05}. In addition, as noted by Kuttler \cite{K72} as well as Wang and Xia \cite{WX09}, the first eigenvalue satisfies an isoperimetric inequality
\begin{align*}
\eta_0\leq \frac{|\pt \Omega|}{|\Omega|}.
\end{align*}
Moreover, Ferrero, Gazzola, and Weth \cite{FGW05} proved that if the equality in the above inequality holds for a Euclidean domain, then the domain must be a ball; cf. Theorem~1.3 in \cite{WX09}. See \cite{RS15} for other sharp bounds on $\eta_0$, and also \cite{BLSV23} for interesting estimates on $\eta_0$ in the setting of free boundary hypersurfaces in a Euclidean ball.

In this paper we are concerned with the shape optimization for the problem~\eqref{problem2}. For this problem, the rescaling property reads $\eta_k((\Omega,c^2g))=c^{-1}\eta_k((\Omega,g))$. So one natural normalization of the eigenvalues is given as $\eta_k(\Omega)|\pt \Omega|^{1/(n-1)}$. Similarly to the problem~\eqref{problem1} of the type one, for the problem~\eqref{problem2} we obtain two main results as in Theorems~\ref{thm3} and \ref{thm4}.
\begin{thm}\label{thm3}
	Let $n\geq 2$ and $k \geq 0$. For the problem \eqref{problem2} on the Euclidean annular domain $\mathbb {B}^{n}_{1} \setminus \overline{\mathbb {B}^{n}_{\epsilon}}$, its $k$th distinct Steklov eigenvalue $\eta_{(k)}$ (the notation for the case $n=2$ or $n=3$ is slightly different; see Remark~\ref{exceptional-notation} below) has the following asymptotic expansion as $\epsilon \to 0$.
	\begin{itemize}
		\item[{\rm (i)}]  When $n+2k \geq 5$,
		\begin{equation*}
			\eta_{(k)} = n+2k
			-\frac{1}{4} (n+2k)(n+2k-2)^{2}\epsilon^{n+2k-3}
			+O\left(\epsilon^{n+2k-2}\right).
		\end{equation*}
		\item[{\rm (ii)}]  When $n=2,~k=0$,
		\begin{equation*}
			\eta^{(1)}_{(0)}
			= 4\epsilon \log^{2}\epsilon+O\left(\epsilon \log\epsilon\right),\quad
			\eta^{(2)}_{(0)} =4+4\epsilon \log^{2}\epsilon +O\left(\epsilon \log \epsilon  \right).
		\end{equation*}
		\item[{\rm (iii)}]  When $n=3,~k=0$,
		\begin{equation*}
			\eta^{(1)}_{(0)}=\frac{2}{1-\epsilon}, \quad
			\eta^{(2)}_{(0)}=\frac{6}{1-\epsilon}.
		\end{equation*}
		\item[{\rm (iv)}]  When $n=2,~k=1$ or $n=4,~k=0$,
		\begin{equation*}
			\eta_{(k)}=4-4\epsilon +O\left(\epsilon \log^{-1}\epsilon \right).
		\end{equation*}
	\end{itemize}
	
	In particular, the first eigenvalue as $\epsilon \to 0$ is given by
	\begin{align*}
		\eta_{(0)} =
\begin{cases}
4\epsilon \log^{2}\epsilon+O\left(\epsilon \log\epsilon\right),\ n=2,\\
\dfrac{2}{1-\epsilon},\ n=3,\\
4-4\epsilon +O\left(\epsilon \log^{-1}\epsilon \right),\ n=4,\\
n - \dfrac{n(n-2)^2}{4}\epsilon^{n-3} + O\left(\epsilon^{n-2}\right),\ n\geq 5.
\end{cases}
	\end{align*}
Moreover, when $n\geq 4$, we have
\begin{align*}
\begin{cases}
\eta_{(k)}\left(\mathbb {B}^{n}_{1} \setminus \overline{\mathbb {B}^{n}_{\epsilon}}\right) \left|  \partial \left( \mathbb {B}^{n}_{1}  \setminus \overline{\mathbb {B}^{n}_{\epsilon}} \right)\right| ^{\frac{1}{n-1}}  <
		\eta_{(k)}\left(\mathbb {B}^{n}_{1}\right) \left|  \partial \left( \mathbb {B}^{n}_{1}\right)\right| ^{\frac{1}{n-1}},\ k= 0, 1,\\
\eta_{(k)}\left(\mathbb {B}^{n}_{1} \setminus \overline{\mathbb {B}^{n}_{\epsilon}}\right) \left|  \partial \left( \mathbb {B}^{n}_{1}  \setminus \overline{\mathbb {B}^{n}_{\epsilon}} \right)\right| ^{\frac{1}{n-1}}  >
		\eta_{(k)}\left(\mathbb {B}^{n}_{1}\right) \left|  \partial \left( \mathbb {B}^{n}_{1}\right)\right| ^{\frac{1}{n-1}},\ k\geq 2.
\end{cases}
\end{align*}
\end{thm}
As before, we get the following conclusion on the shape optimization.
\begin{thm}\label{main-purpose2}
Let $n\geq 4$. For the normalized eigenvalue $\eta_m(\Omega)|\pt \Omega|^{1/(n-1)}$ ($0\leq m\leq n$), the Euclidean balls are not the minimizers among Euclidean domains; for the normalized eigenvalue $\eta_m(\Omega)|\pt \Omega|^{1/(n-1)}$ ($m\geq n+1$), the Euclidean balls are not the maximizers among Euclidean domains.
\end{thm}

\begin{rem}\label{exceptional-notation}
For $n=2$, asymptotically the spectrum is given by a sequence $\{0+,4-,4+,6-,8-,10-,\dots\}$; due to the exceptional eigenvalue $0+$ and $4+$, the $\eta_{(k)}$ in (i) is in fact the $(k+1)$st eigenvalue. For $n=3$, asymptotically the spectrum is given by a sequence $\{2+,5-,6+,7-,9-,\dots\}$; so the $\eta_{(k)}$ in (i) is also the $(k+1)$st eigenvalue. Therefore, for $n=2$ and $3$, it seems little interesting to compare the normalized spectra on $\mathbb {B}^{n}_{1} \setminus \overline{\mathbb {B}^{n}_{\epsilon}}$ and $\mathbb {B}^{n}_{1}$.
\end{rem}
\begin{rem}
In Theorems~3 and 4 of the work \cite{BFG09}, Bucur, Ferrero, and Gazzola studied $\lim_{\epsilon \to 0} \eta_0(\mathbb {B}^{n}_{1} \setminus \overline{\mathbb {B}^{n}_{\epsilon}})$ using a different method. Their results are consistent with ours. See \cite{BG11,AG13,Las17} for other interesting results on the shape optimization of $\eta_0(\Omega)$.
\end{rem}

In the next result we determine explicitly all the eigenvalues and the corresponding eigenfunctions of the cylinder for the problem~\eqref{problem2}.

\begin{thm}\label{thm4}
	Let $C_{L}=[-L,L] \times M$ be a cylinder for $L \in \mathbb{R^{+}}$, where $M$ is a compact Riemannian manifold without boundary. Denote the spectrum of the Laplace--Beltrami operator $\Delta_{M}$ on $M$ by
	\begin{align*}
		0=\lambda_{0} (M)<\lambda_{1} (M) \leq \lambda_{2} (M) \leq \cdots,
	\end{align*}
	and let $\{\beta_k\}^{\infty}_{k=0} \subset L^{2}(M)$ be a corresponding orthonormal basis of eigenfunctions such that
	\begin{align*}
		-\Delta_{\Sigma} \beta_k=\lambda_{k}(M) \beta_k.
	\end{align*}
Let $c_k=\lambda_k(M)^{1/2}$. Then the eigenvalues of the cylinder $C_{L}$ for the fourth-order Steklov eigenvalue problem $\eqref{problem2}$ are given by
	\begin{align*}
		\frac{1}{L}, ~
		\frac{3}{L}, ~
		\frac{2c_k \left(e^{2 L c_k}-1\right)^2}{e^{4Lc_k}-4L c_k e^{2L c_k}-1}, ~
		\frac{2c_k \left(e^{2 L c_k}+1\right)^2}{e^{4Lc_k}+4Lc_k e^{2Lc_k}-1}
	\end{align*}
	and the corresponding eigenfunctions $u(s,x)$ are
	\begin{align*}
		&-L^2 +s^2, ~
		-L^2 s +s^3, 	\\
		&\left(-\frac{Le^{2L c_k}+L}{e^{2L c_k}-1} e^{c_k s}
		+s e^{c_k s}
		+\frac{Le^{2L c_k}+L}{e^{2L c_k}-1} e^{-c_k s}
		+s e^{-c_k s}\right) \beta_k (x) ,\\
		&\left(\frac{Le^{2L c_k}-L}{e^{2L c_k}+1} e^{c_k s}
		-s e^{c_k s}
		+\frac{Le^{2L c_k}-L}{e^{2L c_k}+1} e^{-c_k s}
		+s e^{-c_k s}\right) \beta_k (x).
	\end{align*}
\end{thm}

\begin{rem}\label{remark2}
	We have the following observations.
	\begin{itemize}
		\item[{\rm (i)}]  Consider the asymptotics of the Steklov eigenvalues $\eta_m$ as $m \rightarrow \infty$. The eigenvalues satisfy
		\begin{align*}
			\eta_{m}(C_{L})=2\sqrt{\lambda_{m}(\Sigma)}+O\left(m^{-\infty}\right),\text{ as } m\to \infty.
		\end{align*}
		\item[{\rm (ii)}]  Consider the limiting behaviour as $L \rightarrow 0$. The eigenvalues satisfy
		\begin{equation*}
			\begin{gathered}
				\frac{1}{L} \rightarrow \infty, \frac{3}{L}\rightarrow \infty,
				\frac{2c_k \left(e^{2 L c_k}-1\right)^2}{e^{4Lc_k}-4L c_k e^{2L c_k}-1}
				 \rightarrow  \infty,
				\frac{2c_k \left(e^{2 L c_k}+1\right)^2}{e^{4Lc_k}+4Lc_k e^{2Lc_k}-1} \rightarrow \infty .
			\end{gathered}
		\end{equation*}
		\item[{\rm (iii)}]  Consider the limiting behaviour as $L \rightarrow \infty$. The eigenvalues satisfy
		\begin{equation*}
			\begin{gathered}
				\frac{1}{L} \rightarrow 0,
				 \frac{3}{L}\rightarrow 0,
				\frac{2c_k \left(e^{2 L c_k}-1\right)^2}{e^{4Lc_k}-4L c_k e^{2L c_k}-1}\rightarrow 2c_{k},
				\frac{2c_k \left(e^{2 L c_k}+1\right)^2}{e^{4Lc_k}+4Lc_k e^{2Lc_k}-1} \rightarrow 2c_{k}.
			\end{gathered}
		\end{equation*}
So we see $\eta_m(C_L)\to 2\lambda_m(\Sigma)^{1/2}$ as $L\to \infty$.
	\end{itemize}
\end{rem}
As before, by (ii) of Remark~\ref{remark2}, we get the following result.
\begin{thm}\label{thm-existence2}
 There exist Riemannian manifolds for which the first Steklov eigenvalue $\eta_0$ is arbitrarily large while keeping the boundary area fixed.
\end{thm}
Therefore, in order to study upper bounds of $\eta_0(\Omega)|\pt \Omega|^{1/(n-1)}$ for a general Riemannian manifold, we have to impose certain geometric conditions.

In the third part of this paper, we consider another fourth-order Steklov eigenvalue problem. It was proposed by Buoso and Provenzano \cite{BP15} in 2015, and its definition is as follows.

Let $\Omega$ be a bounded domain in $\R^n$ ($n\geq 2$) of the class $C^{0,1}$. We consider the following fourth-order Steklov eigenvalue problem of the type three:
\begin{equation}\label{problem3}
	\begin{cases}
		\Delta^2 u-\tau \Delta u=0, &\text{in }\Omega,\\
		\dfrac{\pt^2 u}{\pt \nu^2}=0, &\text{on }\pt \Omega,\\
		\tau \dfrac{\pt u}{\pt \nu}-\mathrm{div}_{\pt \Omega}(D^2u(\nu,\cdot))-\dfrac{\partial \Delta u}{\pt \nu}=\rho u, &\text{on }\pt \Omega,
	\end{cases}
\end{equation}
where $\tau>0$ is a parameter and $\nu$ denotes the outward unit normal to $\partial \Omega$. This problem was first studied in \cite{BP15} for domains of the class $C^1$. For domains of the class $C^{0,1}$, by the argument with minor modification as in \cite[Theorem~3.10]{LP22}, the problem \eqref{problem3} is also well defined in the weak sense. Namely, there are non-trivial functions $u\in H^2(\Omega)$ and $\rho \in \R$ such that
\begin{equation}\label{problem3.1}
	\int_{\Omega} \langle D^2 u, D^2 \varphi \rangle  +\tau \langle \nabla u,\nabla \varphi\rangle dx=\rho \int_{\pt \Omega} u \varphi d\sigma,\quad \forall \varphi \in H^2(\Omega).
\end{equation}
The spectrum of the problem \eqref{problem3} or its weak form \eqref{problem3.1} consists of a discrete sequence of eigenvalues tending to the infinity:
\begin{equation}
	0=\rho_{0}(\Omega)<\rho_{1}(\Omega) \leq \rho_{2}(\Omega) \leq \cdots \nearrow +\infty.
\end{equation}
The min-max variational principle for the eigenvalue reads
\begin{align}
	\rho_k(\Omega)=\sup_{V_{k}\subset H^2(\Omega),\ \mathrm{dim}(V_{k})=k} \inf_{0\neq u\perp V_{k}\text{ on }\pt \Omega}\frac{\int_\Omega |D^2 u|^2+\tau |\nabla u|^2 dx}{\int_{\pt \Omega}u^2 d\sigma}.
\end{align}
In particular, we get
\begin{align}\label{variational-characterization}
	\rho_1(\Omega)=\inf_{0\neq u\in H^2(\Omega),\ \int_{\pt \Omega}ud\sigma =0}\frac{\int_\Omega |D^2 u|^2+\tau |\nabla u|^2 dx}{\int_{\pt \Omega}u^2 d\sigma}.
\end{align}

If $\Omega$ is a ball of the radius $r$ in $\mathbb{R}^{n}$, the first non-zero eigenvalue is $\rho_{1}(\Omega)=\tau/r$, proved by Buoso and Provenzano \cite{BP15}. They \cite{BP15} also obtained an isoperimetric inequality for the fundamental tone $\rho_{1}(\Omega)$ which states that
\begin{align*}
	\rho_{1} (\Omega)\leq \rho_{1} (\mathbb{B}_{\Omega}),
\end{align*}
with the equality if and only if $\Omega$ is a ball, where $\mathbb{B}_{\Omega}$ is the ball of the same volume as $\Omega$. In \cite{DMWXZ}, Du, Mao, Wang, Xia, and Zhao considered a more general problem than \eqref{problem3}. As a corollary, they proved an isoperimetric bound for the first non-zero eigenvalue of the problem \eqref{problem3} on bounded domains of a Euclidean space, which reads
\begin{align*}
	\rho_{1}(\Omega) \le \frac{n \tau |\Omega|  \int_{\pt \Omega} H^2 d\sigma }{|\pt \Omega|^2}
\end{align*}
with the equality if and only if $\Omega$ is a ball. Here $H$ is the mean curvature of $\pt \Omega$ in $\mathbb{R}^{n}$. See \cite{DMWXZ} and references therein for more related estimates on $\rho_{k}(\Omega)$.  

In this paper, we provide another sharp upper bound for the fundamental tone of the fourth-order Steklov eigenvalue problem $\eqref{problem3}$ as follows.
\begin{thm}\label{thm5}
	Let $\Omega\subset \R^n$ ($n\geq 2$) be a smooth bounded domain with a star-shaped and mean convex boundary $\pt \Omega$. Then
	\begin{align}
\rho_1(\Omega)\leq n\tau |\Omega||\pt \Omega|\left(\frac{n-1}{n}\omega_{n-1}\left(\frac{|\pt \Omega|}{\omega_{n-1}}\right)^{n/(n-1)}+|\Omega|\right)^{-2},
	\end{align}
	with the equality if and only if $\Omega$ is a ball, where $\omega_{n-1}=|\SS^{n-1}|$.
\end{thm}
\begin{rem}
Using the Young inequality, we may get a better-looking (but weaker) inequality
\begin{align*}
\rho_1(\Omega)\leq \frac{n^{1-2/n}\omega_{n-1}^{2/n}\tau |\Omega|^{1-2/n}}{|\pt \Omega|}.
\end{align*}
\end{rem}
In the above theorem ``star-shaped'' means that the support function of the boundary is positive everywhere. For the proof of Theorem~\ref{thm5}, after translating the origin, we use coordinate functions $x_i$ ($i=1,2,\dots,n$) as test functions in the variational characterization~\eqref{variational-characterization} of $\rho_1(\Omega)$. Then the proof reduces to the estimate on $\int_{\pt \Omega} |x|^2 d\sigma$, and we may employ the result obtained by Kwong and Wei in \cite{KW23} to conclude.


The structure of this paper is as follows. In Section~2, we review the definition and properties of spherical harmonics in terms of the harmonic homogeneous polynomials, and obtain the expression of the radial function part of the separate variable solutions of the biharmonic equation. In Section~3, using the method of separating variables, we establish the asymptotic expansion of the $k$th distinct Steklov eigenvalue for the fourth-order Steklov eigenvalue problem~\eqref{problem1} on the Euclidean annular domain, and determine the eigenvalues and eigenfunctions of the same problem for cylinders. That is, we prove Theorems~\ref{thm1} and \ref{thm2}. In Section~4, we consider the fourth-order Steklov eigenvalue problem~\eqref{problem2}, and get the corresponding results as in Section~3. So we prove Theorems~\ref{thm3} and \ref{thm4}. In Section~5, we deduce a sharp upper bound for the first non-zero eigenvalue of the fourth-order Steklov eigenvalue problem~\eqref{problem3} on a star-shaped and mean convex domain in $\mathbb{R}^{n}$, thus proving Theorem~\ref{thm5}.


\section{Preliminaries}

In this section we provide some properties on the spherical harmonics and determine the radial function part of solutions of separate variables to the biharmonic equation on $\R^n$.

\subsection{Spherical harmonics} Given a spherical harmonic $\beta$ on $\mathbb{S}^{n-1}$ of the degree $k \geq 0$, it can be seen as a restriction on $\mathbb{S}^{n-1}$ of a harmonic homogeneous polynomial $\tilde{\beta}$ on $\mathbb{R}^n$ of the same degree $k$.
Consider the spaces of harmonic homogeneous polynomials on $\mathbb{R}^n$:
\begin{align*}
	\mathcal{D}_{k}:=\{V \in C^{\infty}(\mathbb{R}^n)|\Delta V=0 ~\text{in $\mathbb{R}^n$, $V$ homogeneous polynomial of degree $k$}\}.
\end{align*}
Let $\mu_{k}$ be the dimension of $\mathcal{D}_{k}$. In particular, we know
\begin{align*}
	\mathcal{D}_{0}=&span\{1\}, \quad  \mu_{0}=1, \\
	\mathcal{D}_{1}=&span\{x_{i}, ~i=1, \dots, n\}, \quad  \mu_{1}=n, \\
	\mathcal{D}_{2}=&span\{x_{i} x_{j}, ~x_{1}^{2}-x_{p}^{2},~1 \leq i<j \leq n,~2 \leq p \leq n\}, \quad  \mu_{2}=\frac{n^2+n-2}{2},
\end{align*}
and $\mu_{k}=C^{n-1}_{n+k-1}-C^{n-1}_{n+k-3},~k \geq 2$. Consult \cite{ABR92} for basic facts on $\mathcal{D}_k$ and $\mu_k$.

For a spherical  harmonic $\beta$ on $\mathbb{S}^{n-1}$ of the degree $k \geq 0$,
one of its basic properties is that
$-\Delta_{\mathbb{S}^{n-1}} \beta=\tau_{k} \beta$,
where
$\tau_{k}=k(n+k-2)$.

\subsection{Solutions of separate variables to the biharmonic equation} Next, we will give the expression for the radial function of the separate variable solutions.
\begin{prop}\label{prop1}
	Let
	$u(r,\phi)=\alpha \left ( r \right )\beta \left (\phi  \right )$ ($r\in \R^+$, $\phi\in\SS^{n-1}$)
	be a separate variable solution of the biharmonic equation
	$\Delta^{2}u=0$ on $\R^n$,
	where
	$\beta\left (\phi\right )$
	is a spherical harmonic on
	$\mathbb {S}^{n-1}$
	of the degree $k$,
	that is,
	$\beta $
	satisfies the equation
	$-\Delta_{\mathbb{S}^{n-1}} \beta=\tau_{k} \beta$,
	where
	$\tau_{k}=k(n+k-2)$.
	Then
	$\alpha\left(r\right)$
	is of the form:
	\begin{align*}
		\alpha(r)=
		\begin{cases}
			a+b\log r+cr^{2}+dr^{2}\log r,\quad  n=2,\ k=0,\\
			ar+br^{-1}+cr^{3}+dr\log r,\quad n=2,\ k=1,\\
			a+br^{-2}+cr^{2}+d\log r,\quad n=4,\ k=0,\\
			ar^{k}+br^{2-n-k}+cr^{k+2}+dr^{4-n-k },\quad \text{otherwise},
		\end{cases}		
	\end{align*}
	where $a, b, c, d$ are all constants.
\end{prop}

\begin{proof}
	In polar coordinates, the $n$-dimensional
	Laplacian operator $\Delta$
	can be written as
	\begin{align*}
		\Delta=\frac{\partial^2}{\partial r^2}+\frac{n-1}{r} \frac{\partial}{\partial r}+\frac{1}{r^{2}}\Delta_{\SS^{n-1}}.
	\end{align*}
	Then
	\begin{align*}
		\Delta u=\alpha ^{(2)}\beta +\frac{n-1}{r}\alpha ^{(1)}\beta-\frac{k\left( n+k-2\right )}{r^{2}}\alpha\beta,
	\end{align*}
	and
	\begin{align*}
		\Delta ^{2}u=&\alpha^{(4)} \beta+\frac{2(n-1)}{r} \alpha^{(3)} \beta+\frac{(n-1)(n-3)-2 k(n+k-2)}{r^2} \alpha^{(2)} \beta \\
		&-\frac{(n-1)(n-3)+2k(n-3) (n+k-2)}{r^3} \alpha^{(1)} \beta \\
		&-\frac{k(n+k-2)[(8-2 n)-k(n+k-2)]}{r^4} \alpha \beta \\
		=& 0.
	\end{align*}
	Multiplying the above equation by $r^{4}$, we get an Euler equation
	\begin{align*}
		& \alpha^{(4)} r^4+2(n-1) \alpha^{(3)} r^3+[(n-1)(n-3)-2 k(n+k-2)] \alpha^{(2)} r^2 \\
		-& [2k(n-3) (n+k-2)+(n-1)(n-3)] \alpha^{(1)} r \\
		-& k(n+k-2)[(8-2 n)-k(n+k-2)] \alpha
		= 0.
	\end{align*}
	Let $r=e^{t}$. The Euler equation may be rewritten as
	\begin{align*}	
		& D(D-1)(D-2)(D-3) \alpha+2(n-1) D(D-1)(D-2) \alpha
		\\
&+\left[(n-1)(n-3)\right.
		-\left.2 k(n+k-2)\right] D(D-1) \alpha\\
&
		-\left[2k(n-3) (n+k-2)+(n-1)(n-3)\right] D \alpha \\
		&+ k(n+k-2)\left[2(n-4)+k(n+k-2)\right] \alpha
		=0,
	\end{align*}
	where $D$ represents taking the derivative of $\alpha$
	with respect to $t$.
	
	Next, we arrange the above formula to get
	\begin{align*}
		& D^4 \alpha+2(n-4) D^3 \alpha+\left[n^2-10 n+20-2 k(n+k-2)\right] D^2 \alpha\\
&
		-2(n-4)\left[n
		-2+ k(n+k-2)\right]D \alpha\\
&
		+ k(n+k-2)\left[2(n-4)+k(n+k-2)\right] \alpha
		=0.
	\end{align*}
	Therefore, the characteristic equation of this Euler equation is
	\begin{align*}
		 &x^4+2(n-4) x^3+\left[n^2-10 n+20-2 k(n+k-2) \right]x^2\\
&
		-2(n-4)\left[n-2
		+ k(n+k-2)\right]x\\
&
		+ k(n+k-2)[2(n-4)+k(n+k-2)]
		=0.
	\end{align*}
	Solve the above equation to get
	\begin{align*}
		x_1=k, \qquad x_2=2-n-k,\qquad x_3=k+2,\qquad x_4=4-n-k.
	\end{align*}
	
	Obviously $x_{1} \neq x_{3}$, $x_{2} \neq x_{4}$ and $x_{2} \neq x_{3}$. Next consider the case of multiple roots.
	\begin{itemize}
		\item[{\rm (i)}]  If $x_{1}=x_{2}$ or $x_{3}=x_{4}$, i.e., $n=2$, $k=0$,
		the roots of the Euler equation are
		\begin{align*}
			x_1=0,\qquad x_2=0,\qquad x_3=2, \qquad x_4=2.
		\end{align*}
So the Euler equation has two double roots and
	\begin{align*}
		\alpha \left ( t \right )=a+bt+ce^{2t}+dte^{2t}.
	\end{align*}
	As $t=\log r$, we have
	\begin{align*}
		\alpha \left ( r \right )=a+b\log r+cr^{2}+dr^{2}\log r.
	\end{align*}
    \item[{\rm (ii)}]  If $x_{1}=x_{4}$, i.e., $n+2k=4$, there are two cases as follows.
    \begin{itemize}
    \item[{\rm (a)}] When $n=2$, $k=1$, the roots of the Euler equation are
    \begin{align*}
    	x_1=1,\qquad x_2=-1,\qquad x_3=3, \qquad x_4=1.
    \end{align*}
So the Euler equation has one double root and
    \begin{align*}
    	\alpha \left ( t \right )=ae^{t}+be^{-t}+ce^{3t}+dte^{t}.
    \end{align*}
    As $t=\log r$, we have
    \begin{align*}
    	\alpha \left ( r \right )=ar+br^{-1}+cr^{3}+dr\log r.
    \end{align*}
    \end{itemize}
\begin{itemize}
	\item[{\rm (b)}] When $n=4$, $k=0$, the roots of the Euler equation are
	\begin{align*}
		x_1=0,\qquad x_2=-2,\qquad x_3=2, \qquad x_4=0.
	\end{align*}
So the Euler equation has one double root and
	\begin{align*}
		\alpha \left ( t \right )=a+be^{-2t}+ce^{2t}+dt.
	\end{align*}
	As $t=\log r$, we have
	\begin{align*}
		\alpha \left ( r \right )=a+br^{-2}+cr^{2}+d\log r.
	\end{align*}
\end{itemize}
	\end{itemize}

Finally, if the Euler equation has no multiple roots, we may get
\begin{align*}
	\alpha \left ( t \right )=ae^{kt}+be^{\left ( 2-n-k \right )t}+ce^{\left ( k+2 \right )t}+de^{ \left ( 4-n-k \right )t}.
\end{align*}
As $t=\log r$, we have
\begin{align*}
	\alpha \left ( r \right )=ar^{k}+br^{2-n-k}+cr^{k+2}+dr^{4-n-k }.
\end{align*}
\end{proof}

\section{Results for the fourth-order Steklov eigenvalue problem of  the type one}



\subsection{The asymptotic expansion of the spectrum on annular domains}
In this subsection, we mainly study the asymptotic expansion of the $k$th distinct Steklov eigenvalue $\xi_{(k)}$ for the fourth-order Steklov eigenvalue problem ~\eqref{problem1} on an annular domain $\mathbb{B}^{n}_{1} \setminus \overline{\mathbb{B}^{n}_{\epsilon}}$. Namely, we shall prove Theorem~\ref{thm1}. Besides, note that in the proof of Theorem~\ref{thm1} (and Theorem~\ref{thm3}), for the notational simplicity, we will often omit the limit $\epsilon \to 0$.

\begin{proof}[Proof of Theorem~\ref{thm1}]

In our situation we may employ the method of separating variables. So we assume that the eigenfunction corresponding to the $k$th distinct Steklov eigenvalue $\xi_{(k)}$ is of the form $u(r,\phi)=\alpha(r)\beta(\phi)$. We apply the two boundary conditions of the fourth-order Steklov eigenvalue problem ~\eqref{problem1} to obtain a system of linear equations. For the annular domain $\mathbb{B}^{n}_{1} \setminus \overline{\mathbb{B}^{n}_{\epsilon}}$, we choose the outward unit normal vector $\nu=\partial/\partial r$ and $\nu=-\pt/\pt r$ on $\partial \mathbb{B}^{n}_{1}$ and $\partial \mathbb{B}^{n}_{\epsilon}$, respectively. Thus, the boundary conditions of ~\eqref{problem1} can be written as
\begin{equation}
\begin{aligned}
	\frac{\partial u}{\partial r }\bigg|_{r=1}=&\frac{\partial u}{\partial r }\bigg|_{r=\epsilon }=0,\\
	\left (\frac{\partial\Delta u }{\partial r }+\xi u  \right )\bigg|_{r=1}=&\left (-\frac{\partial\Delta u }{\partial r }+\xi u  \right )\bigg|_{r=\epsilon }=0.
\end{aligned}
\label{boundary1}
\end{equation}
Next we shall discuss the radial function in different cases.

To begin with, we note $\xi_{0}=0$. Here for completeness, we first verify $\xi_{0}=0$ in our setting, in order to illustrate the method.

\textbf{A.1. The case~$n=2, ~k=0$.}
By Proposition $\ref{prop1}$, we know that the separate variable solution of the biharmonic equation is given by
\begin{align*}
	u(r,\phi)=\alpha(r)\beta(\phi)
	=\left(a+b\log r+cr^{2}+dr^{2}\log r\right)\beta(\phi).
\end{align*}

Now we plug the above solution into the boundary conditions \eqref{boundary1} to  obtain the following system of linear equations
\begin{align*}
\left\{\begin{array}{l}
\begin{aligned}
	&b +2c+d=0, \\
	&\epsilon^{-1}b+2\epsilon c+\epsilon\left(2 \log\epsilon +1 \right)d=0,\\
	&\xi a+\xi c+4d=0,\\
	&\xi a+\xi \log\epsilon b+\xi \epsilon^{2}c+\left(\xi \epsilon^{2} \log\epsilon -4\epsilon^{-1}\right)d=0.
\end{aligned}
\end{array}\right.
\end{align*}
So the system of linear equations has non-zero solutions if and only if
\begin{align*}
\left(
-\epsilon^{-1}
+4\epsilon \log^{2}\epsilon
+2\epsilon
-\epsilon^{3}
\right) \xi^{2}+8
\left(\epsilon^{-2}
+\epsilon^{-1}
-\epsilon
-1\right) \xi=0.
\end{align*}
Multiplying it by
$\epsilon^{2}$, we get
\begin{align*}
\left(-\epsilon
+4\epsilon^{3} \log^{2}\epsilon
+2\epsilon^{3}
-\epsilon^{5}
\right) \xi^{2}+8
\left(1+\epsilon
-\epsilon^{2}
-\epsilon^{3}
\right) \xi=0.
\end{align*}
Obviously, the roots of the above equation satisfy
\begin{align*}
\xi^{(1)}_{(0)} =0, \quad
\xi^{(2)}_{(0)} =\infty.
\end{align*}

\textbf{B.1. The case $n=4$, $k=0$.}
By Proposition $\ref{prop1}$, in this case the separate variable solution of the biharmonic equation is
\begin{align*}
	u(r,\phi)=\left(a+br^{-2}+cr^{2}+d\log r\right)\beta(\phi).
\end{align*}
By the boundary conditions \eqref{boundary1}, we obtain the following system of linear equations
$$
\left\{\begin{array}{l}
	\begin{aligned}
		&-2b +2c+d=0, \\
		&-2\epsilon^{-3}b+2\epsilon c+\epsilon^{-1}d=0,\\
		&\xi a+\xi b+\xi c-4d=0,\\
		&\xi a+\epsilon^{-2} \xi b+\epsilon^{2} \xi c+\left(\xi \log\epsilon +4\epsilon^{-3}\right)d=0.
	\end{aligned}
\end{array}\right.
$$
So it has non-zero solutions if and only if
\begin{align*}
	4\left(\epsilon^{-3}
	-2\epsilon^{-1}
	+\epsilon
	+\epsilon^{-3} \log\epsilon
	-\epsilon \log\epsilon \right) \xi^{2}+16
	\left(\epsilon^{-6}
	+\epsilon^{-3}
	-\epsilon^{-2}
	-\epsilon\right) \xi=0,
\end{align*}
and  multiplying it by
$\epsilon^{6}$, we get
\begin{align*}
	\left(
	\epsilon^{3} \log\epsilon
	+\epsilon^{3}
	-2\epsilon^{5}
	-\epsilon^{7} \log\epsilon
	+\epsilon^{7}
	\right) \xi^{2}+4
	\left(1
	+\epsilon^{3}
	-\epsilon^{4}
	-\epsilon^{7}\right) \xi=0.
\end{align*}
Obviously, the roots of the above equation satisfy
\begin{align*}
	\xi^{(1)}_{(0)} =0, \quad
	\xi^{(2)}_{(0)} =\infty.
\end{align*}

\textbf{C.1. The case $n=3$, $k=0$ or $n \geq 5$, $k=0$.}
In these cases, we know
\begin{align*}
	u(r,\phi)
	=(a+r^{2-n}b+r^{2}c+r^{4-n}d)\beta(\phi),
\end{align*}
and
\begin{align*}
	\Delta u
	=(2nc+(8-2n)r^{2-n}d)\beta(\phi).
\end{align*}
Similarly, we obtain a system of linear equations with $a,b,c,d$ as the independent variables:
\begin{equation*}
	\left\{\begin{array}{l}
		\begin{aligned}
			&\left( 2-n\right)b +2c+\left(4-n \right)d=0, \\
			&\left(2-n\right)\epsilon ^{1-n}b+2\epsilon c+\left( 4-n\right)\epsilon ^{3-n}d=0,\\
			&\xi a+\xi b+\xi c+\left[\xi + \left( 2-n \right) \left( 8-2n\right)\right]d=0,\\
			&\xi a+\xi \epsilon ^{2-n} b+\epsilon ^{2}\xi c
			+\left[ \epsilon ^{4-n}\xi +\left( n-2 \right) \left( 8-2n\right)\epsilon ^{1-n}\right]d=0.
		\end{aligned}
	\end{array}\right.
\end{equation*}
So
\begin{align*}
&\left[4\epsilon^{3}
-(n-2)^2\epsilon^{n-1}
+2n(n-4)\epsilon^{n+1}
-(n-2)^2\epsilon^{n+3}
+4\epsilon^{2n-1}\right] \xi^{2}\\
&\qquad -4(n-4)(n-2)^2\left(1
+\epsilon^{n-1}
-\epsilon^{n}
- \epsilon^{2n-1}\right) \xi=0.
\end{align*}
Obviously, the roots of the above equation satisfy
\begin{align*}
	\xi^{(1)}_{(0)} =0, \quad
	\xi^{(2)}_{(0)} =\infty.
\end{align*}

Next, we begin to give the proof of Theorem~\ref{thm1}. We shall consider the asymptotic expansion of non-zero eigenvalues $\xi_{(k)}$ for any positive integer $k \geq 1$.

\textbf{A.2. The case $n+2k\geq5$, $k\geq 1$.}
In this case, the solution of $\Delta^{2} u=0$ is
\begin{align*}
u(r,\phi)=\left(ar^{k}+br^{2-n-k}+cr^{k+2}+dr^{4-n-k }\right)\beta(\phi).
\end{align*}
Obviously,
\begin{align}
	&\qquad \qquad \frac{\partial u}{\partial r}=\nonumber \\
&\left[
	k r^{k-1} a
	+(2-n-k)r^{1-n-k} b
	+(k+2)r^{k+1} c+(4-n-k)r^{3-n-k }d\right]\beta(\phi),
	\label{1}\\
	&\qquad \qquad   \Delta u=\left(2(n+2k)r^{k} c+(8-2n-4k)r^{2-n-k} d\right)\beta(\phi),
	\label{2}\\
	&\frac{\partial \Delta u}{\partial r}=\left(2k(n+2k)r^{k-1}c+(8-2n-4k)(2-n-k)r^{1-n-k}d\right)\beta(\phi). \label{3}
\end{align}
By the boundary conditions \eqref{boundary1} and the equations \eqref{1}, \eqref{2} and \eqref{3},  we obtain a system of linear equations with $a,b,c,d$ as the independent variables:
\begin{equation}
\left\{\begin{array}{l}
	\begin{aligned}
		&ka+\left( 2-n-k \right)b +\left( k+2 \right)c+\left( 4-n-k \right)d=0, \\
		&k\epsilon ^{k-1}a+\left( 2-n-k \right)\epsilon ^{1-n-k}b+\left( k+2 \right)\epsilon ^{k+1}c+\left( 4-n-k \right)\epsilon ^{3-n-k}d=0,\\
		&\xi a+\xi b+\left[ \xi +2k\left( n+2k \right)  \right]c+\left[\xi + \left( 2-n-k \right) \left( 8-2n-4k \right)\right]d=0,\\
		&\xi \epsilon ^{k}a+\xi\epsilon ^{2-n-k} b+\left [ \epsilon ^{k+2}\xi -2k\left ( n+2k \right )\epsilon ^{k-1} \right ]c\\
		&+\left[ \epsilon ^{4-n-k}\xi +\left( n+k-2 \right) \left( 8-2n-4k \right)\epsilon ^{1-n-k}\right]d=0.
	\end{aligned}
\end{array}\right.
\label{linear equ}
\end{equation}
The sufficient and necessary condition for a system of homogeneous linear equations to have non-zero solutions is that the determinant of the coefficient matrix is equal to zero. So we get
	\begin{align*}
		&0=\\
		&\begin{vmatrix}
			k & 2-n-k & k+2 & 4-n-k\\
			k\epsilon ^{k-1} & \left( 2-n-k \right)\epsilon^{1-n-k} & \left( k+2 \right)\epsilon^{k+1} & \left( 4-n-k \right)\epsilon^{3-n-k}\\
			 \xi  & \xi  &  \xi +2k\left( n+2k \right)   &  a_{34}(\xi)\\
			\xi \epsilon^{k} & \xi\epsilon^{2-n-k} & \epsilon^{k+2}\xi -2k\left ( n+2k \right )\epsilon^{k-1} & a_{44}(\xi)
		\end{vmatrix}\\
		&=A\xi^2+B\xi+C=:f(\xi),
	\end{align*}
where
\begin{align*}
	a_{34}(\xi)&:=\xi + 2\left( n+k-2 \right) \left( n+2k-4 \right),\\
	a_{44}(\xi):=&\epsilon^{4-n-k}\xi -2\left( n+k-2 \right) \left( n+2k-4 \right)\epsilon^{1-n-k}.
\end{align*}

We can find $A,B,C$ as follows:
\begin{align*}
	C=&f(0)\\
	=&k(n+k-2)\times(\epsilon^{k-1}-\epsilon^{1-n-k})\\
	&~\times\begin{vmatrix}
		2k (n+2k)  &  2(n+k-2)(n+2k-4) \\
		-2k(n+2k) \epsilon ^{k-1} & -2(n+k-2)(n+2k-4) \epsilon ^{1-n-k}
	\end{vmatrix}\\
    =&4k^2 (n+2k)(n+2k-4)(n+k-2)^2\left( \epsilon^{2k-2}
    -2\epsilon^{-n}+\epsilon^{2-2n-2k} \right),
\end{align*}
\begin{align*}
	&B=f'(0)\\
	&=\begin{vmatrix}
			k & 2-n-k & k+2 & 4-n-k\\
			k\epsilon ^{k-1} & \left( 2-n-k \right)\epsilon^{1-n-k} & \left( k+2 \right)\epsilon^{k+1} & \left( 4-n-k \right)\epsilon^{3-n-k}\\
			 1  & 1  &  1 &  1\\
			0 & 0 & -2k\left ( n+2k \right )\epsilon^{k-1} & a_{44}(0)
		\end{vmatrix}\\
    &~+\begin{vmatrix}
			k & 2-n-k & k+2 & 4-n-k\\
			k\epsilon ^{k-1} & \left( 2-n-k \right)\epsilon^{1-n-k} & \left( k+2 \right)\epsilon^{k+1} & \left( 4-n-k \right)\epsilon^{3-n-k}\\
			 0  & 0  &   2k\left( n+2k \right)   &  2\left( n+k-2 \right) \left( n+2k-4 \right)\\
			\epsilon^{k} &  \epsilon^{2-n-k} & \epsilon^{k+2} & \epsilon^{4-n-k}
		\end{vmatrix}\\
	&={4}(n+2k-4)(n+k-2)^2 \epsilon^{2k+1}\\
		&~+4k(n+2k)(n+2k-4)(n+k-2)\epsilon^{3-n} \\
		&~+4(n+2k-2)^2 [ (k+1)n +k^2-2k-4 ] \epsilon^{2-n}\\
		&~-4(n+2k-2)^2 [ (k+1)n +k^2-2k-4 ]\epsilon^{1-n} \\
		&~-4k(n+2k)(n+2k-4)(n+k-2)\epsilon^{-n}
		-4k^2(n+2k)\epsilon^{5-2n-2k} \\
		&~-4(n+2k-4)(n+k-2)^2\epsilon^{2-2n-2k}
		+4k^2 (n+2k) \epsilon^{2k-2},
\end{align*}
and
\begin{align*}
	A&=f''(0)/2\\
	&=\begin{vmatrix}
			k & 2-n-k & k+2 & 4-n-k\\
			k\varepsilon ^{k-1} & \left( 2-n-k \right)\varepsilon^{1-n-k} & \left( k+2 \right)\varepsilon^{k+1} & \left( 4-n-k \right)\varepsilon^{3-n-k}\\
			 1  & 1  &  1 &  1\\
			\varepsilon^{k} &  \varepsilon^{2-n-k} & \varepsilon^{k+2} & \varepsilon^{4-n-k}
		\end{vmatrix}\\
	&=4\epsilon^{2k+1}
	+4\epsilon^{5-2n-2k}
	-(n+2k-2)^2\epsilon^{5-n}\\
	&~+2(n+2k-4)(n+2k)\epsilon^{3-n}
	-(n+2k-2)^2\epsilon^{1-n}.
\end{align*}
	
Multiplying $A,B,C$ by
  $\epsilon^{2n+2k-2}$, and denoting the resulting quantities by the same symbols, we get
  \begin{align*}
  	A=&
  	4\epsilon^{3}
  	-(n+2k-2)^2\epsilon^{n+2k-1}\\
  	&+2(n+2k)(n+2k-4)\epsilon^{n+2k+1}
  	-(n+2k-2)^2\epsilon^{n+2k+3}
  	+4\epsilon^{2n+4k-1},\\
  	B=&
  	-4(n+2k-4)(n+k-2)^2
  	-4k^2(n+2k)\epsilon^{3}\\
  	&-4k(n+2k)(n+2k-4)(n+k-2)\epsilon^{n+2k-2}\\
  	&-4(n+2k-2)^2 [ (k+1)n +k^2-2k-4 ]\epsilon^{n+2k-1}\\
  	&+4(n+2k-2)^2 \left[(k+1)n +k^2-2k-4\right]\epsilon^{n+2k}\\
  	&+4k(n+2k)(n+2k-4)(n+k-2)\epsilon^{n+2k+1}\\
  	&+4k^2(n+2k) \epsilon^{2n+4k-4}
  	+4(n+2k-4)(n+k-2)^2 \epsilon^{2n+4k-1},\\
  	C=&
  	4k^2 (n+2k)(n+2k-4)(n+k-2)^2\left( 1
  	-2\epsilon^{n+2k-2}+\epsilon^{2n+4k-4}\right).
  \end{align*}
So we have
  \begin{align*}
	&B^2-4AC\\
	=& P(n,k,\epsilon)^2+D(n,k)\epsilon^{n+2k-2}
	+E(n,k)\epsilon^{n+2k-1}+F(n,k)\epsilon^{n+2k}\\
	&+G(n,k)\epsilon^{n+2k+1}
	+H(n,k)\epsilon^{2n+4k-4}
	+O\left(\epsilon^{n+2k+2}\right),
\end{align*}
where
\begin{align*}
P(n,k,\epsilon):=&4(n+2k-4)(n+k-2)^2-4k^2(n+2k)\epsilon^{3},\\
D(n,k):=&2\cdot4^2 k(n+2k)(n+2k-4)^2(n+k-2)^3,\\
	E(n,k):=&4^2(n+2k-4)(n+2k-2)^2 (n+k-2)^2\\ &\cdot \left[(k^2+2k+2)n+2(k^3+k^2 -2k-4)\right],\\
	F(n,k):=&-2\cdot4^2 (n+2k-4)(n+2k-2)^2 (n+k-2)^2 \\
	&\cdot \left[(k+1)n+k^2 -2k-4\right],\\
    G(n,k):=&2\cdot4^2 \left[
    k^3 (n+2k)^2 (n+2k-4)  (n+k-2)\right.\\
    &\left.~-k(n+2k)(n+2k-4)^2(n+k-2)^3 \right.\\
    &\left.~+4k^2 (n+2k) (n+2k-4)(n+k-2)^2\right.\\
    &\left.~-k^2 (n+2k)^2 (n+2k-4)^2 (n+k-2)^2 \right], \\
    H(n,k):=&4^2  k^2  (n+2k)^2  (n+2k-4)^2 (n+k-2)^2\\
    -&2\cdot4^2 k^2 (n+2k) (n+2k-4) (n+k-2)^2.
\end{align*}
Then we obtain
\begin{align*}
	&\sqrt{B^2-4AC}=P(n,k,\epsilon)\\
	&\cdot \left\{
	1+\frac{D(n,k)}{P(n,k,\epsilon)^2}\epsilon^{n+2k-2}+\frac{E(n,k)}
	{P(n,k,\epsilon)^2}
	\epsilon^{n+2k-1}+\frac{F(n,k)}{P(n,k,\epsilon)^2}\epsilon^{n+2k}\right.\\
	&\left.~+\frac{G(n,k)}
	{P(n,k,\epsilon)^2}\epsilon^{n+2k+1}+\frac{H(n,k)}
	{P(n,k,\epsilon)^2}\epsilon^{2n+4k-4}
	+O\left(\epsilon^{n+2k+2}\right)
	\right\}^{1/2}\\
&=P(n,k,\epsilon)\cdot \left\{
	1+\frac{D(n,k)}{2P(n,k,\epsilon)^2}\epsilon^{n+2k-2}+\frac{E(n,k)}
	{2P(n,k,\epsilon)^2}
	\epsilon^{n+2k-1}\right.\\
&+\frac{F(n,k)}{2P(n,k,\epsilon)^2}\epsilon^{n+2k}+\frac{G(n,k)}
	{2P(n,k,\epsilon)^2}\epsilon^{n+2k+1}\\
&\left.+\left(\frac{H(n,k)}
	{2P(n,k,\epsilon)^2}-\frac{1}{8}\frac{D(n,k)^2}{P(n,k,\epsilon)^4}\right)\epsilon^{2n+4k-4}
	+O\left(\epsilon^{n+2k+2}\right)
	\right\}.
\end{align*}

By the expansion
\begin{align*}
\frac{1}{1-x}=1+x+x^2+\cdots,
\end{align*}
we deduce
 \begin{align*}
&\frac{1}{P(n,k,\epsilon)}
	=\frac{1}
	{4(n+2k-4)(n+k-2)^2}
	\cdot \left(1-\frac{k^2 (n+2k)}{(n+2k-4)(n+k-2)^2}\epsilon^3\right)^{-1}\\
	&=\frac{1}
	{4(n+2k-4)(n+k-2)^2}
	\cdot \left[
	1
	+\frac{k^2 (n+2k)}{(n+2k-4)(n+k-2)^2}\epsilon^3
	+O\left(\epsilon^{6}\right)
	\right].
\end{align*}

Therefore, we can derive
\begin{align*}
	&\sqrt{B^2-4AC}=P(n,k,\epsilon)+\frac{D(n,k)}{2P(n,k,\epsilon)}\epsilon^{n+2k-2}+\frac{E(n,k)}
	{2P(n,k,\epsilon)}
	\epsilon^{n+2k-1}\\
&+\frac{F(n,k)}{2P(n,k,\epsilon)}\epsilon^{n+2k}+\frac{G(n,k)}
	{2P(n,k,\epsilon)}\epsilon^{n+2k+1}\\
&+\left(\frac{H(n,k)}
	{2P(n,k,\epsilon)}-\frac{1}{8}\frac{D(n,k)^2}{P(n,k,\epsilon)^3}\right)\epsilon^{2n+4k-4}
	+O\left(\epsilon^{n+2k+2}\right)\\
&=P(n,k,\epsilon)+\frac{D(n,k)}{2P(n,k,0)}\epsilon^{n+2k-2}+\frac{E(n,k)}
	{2P(n,k,0)}
	\epsilon^{n+2k-1}+\frac{F(n,k)}{2P(n,k,0)}\epsilon^{n+2k}\\
&+\left(\frac{G(n,k)}
	{2P(n,k,0)}+\frac{D(n,k)}{2}\frac{k^2 (n+2k)}{4(n+2k-4)^2(n+k-2)^4}   \right)\epsilon^{n+2k+1}\\
&+\left(\frac{H(n,k)}
	{2P(n,k,0)}-\frac{1}{8}\frac{D(n,k)^2}{P(n,k,0)^3}\right)\epsilon^{2n+4k-4}
	+O\left(\epsilon^{n+2k+2}\right).
\end{align*}
Simplifying the above expression, we get
\begin{equation*}
\begin{aligned}
&\sqrt{B^2-4AC}\\
=& \ 4(n+2k-4)(n+k-2)^2
-4k^2(n+2k)\epsilon^{3}\\
&+4k(n+2k)(n+2k-4)(n+k-2)\epsilon^{n+2k-2}\\
&+2(n+2k-2)^2 \left[(k^2+2k+2)n+2(k^3+k^2-2k-4)\right]\epsilon^{n+2k-1}\\
&-4(n+2k-2)^2 \left[(k+1)n+k^2-2k-4\right]\epsilon^{n+2k}\\
&-\left[4k(n+2k)(n+2k-4)(n+k-2)+4k^2 (n+2k)^2 (n+2k-4)\right.\\
&\left. -\frac{8k^3(n+2k)^2}{n+k-2}-16 k^2 (n+2k)\right]\epsilon^{n+2k+1}
-4 k^2 (n+2k) \epsilon^{2n+4k-4}\\
&+O\left(\epsilon^{n+2k+2}\right).
\label{sqrtb4ac}
\end{aligned}
\end{equation*}

Besides, we have
\begin{equation*}
\begin{aligned}
	\frac{1}{A}
	=& \frac{1}{4\epsilon^{3}\left[1-Q(n,k,\epsilon)\right]}\\
	=& \frac{1}{4}\epsilon^{-3}\left[1+Q(n,k,\epsilon)+Q^2(n,k,\epsilon)+\cdots\right]\\
	=& \frac{1}{4}\left[\epsilon^{-3}
	+\frac{1}{4}(n+2k-2)^2 \epsilon^{n+2k-7}
	-\frac{1}{2}(n+2k)(n+2k-4)\epsilon^{n+2k-5}\right.\\
	&\left. +\frac{1}{16}(n+2k-4)^4 \epsilon^{2n+4k-11}
	+O\left(\epsilon^{n+2k-4}\right)\right],
\end{aligned}\label{1/a}
\end{equation*}
where
\begin{align*}
	Q(n,k,\epsilon)
	=&\frac{1}{4}(n+2k-2)^2\epsilon^{n+2k-4}
	-\frac{1}{2}(n+2k)(n+2k-4)\epsilon^{n+2k-2}\\
	+&\frac{1}{4}(n+2k-2)^2\epsilon^{n+2k}
	-\epsilon^{2n+4k-4}.
\end{align*}
Thus, we may calculate that the quadratic equation for $\xi$ has roots
\begin{align*}
	\xi&=\frac{1}{8} \left[
	\epsilon^{-3}+
	\frac{1}{4}(n+2k-2)^2 \epsilon^{n+2k-7}
	-\frac{1}{2}(n+2k)(n+2k-4)\epsilon^{n+2k-5}\right.\\
	&\left. +\frac{1}{16}(n+2k-4)^4 \epsilon^{2n+4k-11}+ O\left(\epsilon^{n+2k-4}\right)
	\right] \\
	&\times \left\{4(n+2k-4)(n+k-2)^2
	+4k^2(n+2k)\epsilon^{3}\right.\\
	&\left.~+4k(n+2k)(n+2k-4)(n+k-2)\epsilon^{n+2k-2}\right.\\
	&\left.~+4(n+2k-2)^2 \left[(k+1)n+k^2-2k-4\right]\epsilon^{n+2k-1}\right.\\
	&\left.~-4(n+2k-2)^2 \left[(k+1)n+k^2-2k-4\right]\epsilon^{n+2k}\right.\\
	&\left.~-4k(n+2k)(n+2k-4)(n+k-2)\epsilon^{n+2k+1}\right.\\
	&\left.~-4k^2(n+2k) \epsilon^{2n+4k-4}
	-4(n+2k-4)(n+k-2)^2 \epsilon^{2n+4k-1}\right.\\
	&\left.~ \pm \left[
	4(n+2k-4)(n+k-2)^2
	-4k^2(n+2k)\epsilon^{3}\right.\right.\\
	&\left.\left. \quad+4k(n+2k)(n+2k-4)(n+k-2)\epsilon^{n+2k-2}\right.\right.\\
	&\left.\left. \quad+2(n+2k-2)^2 \left((k^2+2k+2)n+2(k^3+k^2-2k-4)\right)\epsilon^{n+2k-1}\right.\right.\\
	&\left.\left. \quad-4(n+2k-2)^2 \left((k+1)n+k^2-2k-4\right)\epsilon^{n+2k}\right.\right.\\
	&\left.\left. \quad-\left(4k(n+2k)(n+2k-4)(n+k-2)+4k^2 (n+2k)^2 (n+2k-4)\right.\right.\right.\\
	&\left.\left.\left. \quad-8k^3(n+2k)^2(n+k-2)^{-1}
	-16 k^2 (n+2k)\right)\epsilon^{n+2k+1}\right.\right.\\
	&\left.\left. \quad-4 k^2 (n+2k) \epsilon^{2n+4k-4}+O\left(\epsilon^{n+2k+2}\right)
	\right]
	\right\}.
\end{align*}

To proceed, we need to compare the orders of the terms $\epsilon^{n+2k-5}$ and $\epsilon^{2n+4k-11}$ in the first two lines of the above expression for $\xi$. Thus, we have two cases.
\begin{itemize}
		\item[{\rm (i)}]  	The case $n+2k \geq 7$. The roots of the equation are
\begin{align*}
	\xi^{(1)}_{(k)}
	=&\frac{1}{8}
	\left[
	\epsilon^{-3}
	+\frac{1}{4}(n+2k-2)^2 \epsilon^{n+2k-7}
	-\frac{1}{2}(n+2k)(n+2k-4)\epsilon^{n+2k-5}\right.\\
	&\left.~+O\left(\epsilon^{n+2k-4}\right)\right]
	\cdot\left\{
	8k^2(n+2k)\epsilon^{3}
	-2k^2 (n+2k)(n+2k-2)^2\epsilon^{n+2k-1}\right.\\
	&\left.~+\left[4k^2 (n+2k)^2 (n+2k-4)-8k^3(n+2k)^2(n+k-2)^{-1}\right.\right.\\
	&\left.\left.~-16 k^2 (n+2k)\right]\epsilon^{n+2k+1}
	+O\left(\epsilon^{n+2k+2}\right)
	\right\}\\	
	=&k^2(n+2k)
	-k^2 (n+2k)\left[k(n+2k)(n+k-2)^{-1}+2\right]\epsilon^{n+2k-2}\\
	&~+O\left(\epsilon^{n+2k-1}\right),\\
	\xi^{(2)}_{(k)} =&\infty.
\end{align*}

\item[{\rm (ii)}] The case of $n+2k \leq 6$. We have two sub-cases.
\begin{itemize}
	\item[{\rm (a)}]  The case $n+2k=6$, i.e., $n=2,~k=2$ or $n=4,~k=1$. We get
	\begin{align*}
		\xi^{(1)}_{(k)}=&\frac{1}{8}
	\left[
	\epsilon^{-3}
	+\frac{1}{4}(n+2k-2)^2 \epsilon^{n+2k-7}
    +\left(\frac{1}{16}(n+2k-4)^4\right.\right.\\
	&\left.\left.~-\frac{1}{2} (n+2k)(n+2k-4)\right)\epsilon^{n+2k-5}+O\left(\epsilon^{n+2k-4}\right)\right]\\
	&\cdot\left\{
	8k^2(n+2k)\epsilon^{3}
	-2k^2 (n+2k)(n+2k-2)^2\epsilon^{n+2k-1}\right.\\
	&\left.~+\left[4k^2 (n+2k)^2 (n+2k-4)
	-8k^3(n+2k)^2(n+k-2)^{-1}
	\right.\right.\\
	&\left.\left.~-16 k^2 (n+2k)\right]\epsilon^{n+2k+1}
	+O\left(\epsilon^{n+2k+2}\right)
	\right\}\\	
		=&k^{2}(n+2k)
		-k^{2}(n+2k)\left[\frac{k(n+2k)}{n+k-2}+\frac{1}{2}(n+2k-3)^{3}\right.\\
		&\left.~+\frac{1}{2}(n+2k+1)\right]\epsilon^{n+2k-2} +O\left(\epsilon^{n+2k-1}\right),\\
		\xi^{(2)}_{(k)} =&\infty.
	\end{align*}
	\item[{\rm (b)}]  The case $n+2k=5$, i.e., $n=3,~k=1$. We get
	\begin{align*}
\xi^{(1)}_{(k)}=&\frac{1}{8}
	\left[
	\epsilon^{-3}
	+\frac{1}{4}(n+2k-2)^2 \epsilon^{n+2k-7}+\frac{1}{16}(n+2k-4)^4\epsilon^{n+2k-6}\right.\\
&\left.
	-\frac{1}{2}(n+2k)(n+2k-4)\epsilon^{n+2k-5}+O\left(\epsilon^{n+2k-4}\right)\right]\\
	&\cdot\left\{
	8k^2(n+2k)\epsilon^{3}
	-2k^2 (n+2k)(n+2k-2)^2\epsilon^{n+2k-1}
	\right.\\
	&\left.~+\left[4k^2 (n+2k)^2 (n+2k-4)
	-8k^3(n+2k)^2(n+k-2)^{-1}
	\right.\right.\\
	&\left.\left.~-16 k^2 (n+2k)\right]\epsilon^{n+2k+1}
	+O\left(\epsilon^{n+2k+2}\right)
	\right\}\\	
		=&k^{2}(n+2k)
		-k^{2}(n+2k)\left[\frac{k(n+2k)}{n+k-2}+2\right.\\
		&\left.~+\frac{1}{64}(n+2k-2)^{2} (n+2k-4)^{4} \right]\epsilon^{n+2k-2} +O\left(\epsilon^{n+2k-1}\right),\\
		\xi^{(2)}_{(k)} =&\infty.
	\end{align*}
\end{itemize}
\end{itemize}

\textbf{B.2.~The remaining~case~$n=2,~k=1$.} In this case we have
\begin{align*}
u(r,\phi)&=\left(ar+br^{-1}+cr^3+dr\log r\right)\beta(\phi),\\
\frac{\pt u}{\pt r}&=\left(a-br^{-2}+3cr^2+d(\log r+1)\right)\beta(\phi),\\
\Delta u&=\left(8cr+2dr^{-1}\right)\beta(\phi).
\end{align*}
By the boundary conditions
\begin{align*}
	\frac{\partial u}{\partial r }\bigg|_{r=1}=&\frac{\partial u}{\partial r }\bigg|_{r=\epsilon }=0,\\
	\left (\frac{\partial\Delta u }{\partial r }+\xi u  \right )\bigg|_{r=1}=&\left (-\frac{\partial\Delta u }{\partial  r}+\xi u  \right )\bigg|_{r=\epsilon }=0,
\end{align*}
we obtain a new system of linear equations with $a,b,c,d$ as the independent variables:
$$
\left\{\begin{array}{l}
	\begin{aligned}
		&a-b +3c+d=0, \\
		&a-\epsilon^{-2}b+3\epsilon^{2}c+\left(1+\log\epsilon \right)d=0,\\
		&\xi a+\xi b+\left(\xi +8\right)c-2d=0,\\
		&\xi \epsilon a+\xi \epsilon^{-1} b+\left(\xi \epsilon^{3}-8\right)c+\left(\xi \epsilon \log\epsilon +2\epsilon^{-2}\right)d=0.
	\end{aligned}
\end{array}\right.
$$
After a series of calculations, we get a quadratic equation with respect to $\xi$
\begin{align}
	f(\xi):=&A\xi^2 +B\xi +C \nonumber\\
	=&\left(
	\epsilon^{3}\log\epsilon
	+\epsilon^{3}
	-2\epsilon^{5}
	-\epsilon^{7}\log\epsilon
	+\epsilon^{7}
	\right) \xi^{2}\nonumber\\
&+
	\left(1
	+4\epsilon^{2}
	-4\epsilon^{3}\log\epsilon
	+5\epsilon^{3}-4\epsilon^{4}\log\epsilon
	-5\epsilon^{4}
	-4\epsilon^{5}
	-\epsilon^{7}
	\right) \xi\nonumber\\
&
	+4\left(-1+2\epsilon^{2}-\epsilon^{4}\right)\nonumber\\
	=&0.
\label{f}
\end{align}
Analogously, we find the following expressions
\begin{equation}
	\begin{aligned}
		&\sqrt{B^2-4AC}\\
		=&\sqrt{\left(1+4\epsilon^{2}\right)^2
			+8\epsilon^{3}\log\epsilon
			+26\epsilon^{3}
			-8\epsilon^{4}\log\epsilon
			-10\epsilon^{4}
			-64\epsilon^{5}\log\epsilon
			+O\left(\epsilon^{5}\right)}\\
		=&\left(1+4\epsilon^{2}\right)
		\cdot
		\left[1+\frac{1}{2} \left(
		\frac{8\epsilon^{3}\log\epsilon}{\left(1+4\epsilon^{2}\right)^2}
		+\frac{26\epsilon^{3}}{\left(1+4\epsilon^{2}\right)^2}
		-\frac{8\epsilon^{4}\log\epsilon}{\left(1+4\epsilon^{2}\right)^2}
		-\frac{10\epsilon^{4}}{\left(1+4\epsilon^{2}\right)^2}\right.\right.\\
		&\left.\left.\quad-\frac{64\epsilon^{5}\log\epsilon}{\left(1+4\epsilon^{2}\right)^2}\right)+O\left(\epsilon^{5}\right)
		\right]\\
		=&1
		+4\epsilon^{2}
		+4\epsilon^{3}\log\epsilon
		+13\epsilon^{3}
		-4\epsilon^{4}\log\epsilon
		-5\epsilon^{4}
		-48\epsilon^{5}\log\epsilon
		+O\left(\epsilon^{5}\right),
	\end{aligned}
	\label{b-ac}
\end{equation}
and
\begin{align}\label{1/2a}
		\frac{1}{2A}
		&=\frac{1}{2\epsilon^{3}(\log\epsilon+1)}\frac{1}{1-2\epsilon^2(\log \epsilon+1)^{-1}+O(\epsilon^4)}\nonumber \\
		&=\frac{1}{2\epsilon^{3}(\log\epsilon+1)}(1+2\epsilon^2(\log \epsilon+1)^{-1}+O(\epsilon^4)).
\end{align}
According to the formulas \eqref{b-ac} and \eqref{1/2a}, we can obtain the root of the equation \eqref{f}, i.e.,
\begin{align*}
	\xi=&\frac{1}{2\epsilon^{3}(\log\epsilon+1)}(1+2\epsilon^2(\log \epsilon+1)^{-1}+O(\epsilon^4))\\
	&\times
	\left[
	-1
	-4\epsilon^{2}
	+4\epsilon^{3}\log\epsilon
	-5\epsilon^{3}
	+4\epsilon^{4}\log\epsilon
	+5\epsilon^{4}
	+4\epsilon^{5}
	+\epsilon^{7}\right.\\
	&\left.~\pm
	\left(
	1
	+4\epsilon^{2}
	+4\epsilon^{3}\log\epsilon
	+13\epsilon^{3}
	-4\epsilon^{4}\log\epsilon
	-5\epsilon^{4}
	-48\epsilon^{5}\log\epsilon
	+O\left(\epsilon^{5}\right)
	\right)
	\right].
\end{align*}
Thus, we get
\begin{align*}
	\xi^{(1)}_{(1)} =&\frac{1}{2\epsilon^{3}(\log\epsilon+1)}(1+2\epsilon^2(\log \epsilon+1)^{-1}+O(\epsilon^4))\\
	&\times\left(8\epsilon^3(\log \epsilon+1)-48\epsilon^5 \log \epsilon+O(\epsilon^5)\right)\\
=& 4-24\epsilon^{2}+O\left(\epsilon^{2} \log^{-1}\epsilon\right),\\
	\xi^{(2)}_{(1)} =&\infty.
\end{align*}

Last, we have
\begin{align*}
	\left|  \partial \left( \mathbb {B}^{n}_{1}  \setminus \overline{\mathbb {B}^{n}_{\epsilon}} \right)\right| ^{\frac{3}{n-1}}
	 =&
		\left|  \partial \mathbb {B}^{n}_{1} \right| ^{\frac{3}{n-1}}
		(1+\epsilon^{n-1})^{\frac{3}{n-1}} \\
	=&
	\left|  \partial \mathbb {B}^{n}_{1} \right| ^{\frac{3}{n-1}} \left(1+\frac{3}{n-1}\epsilon^{n-1}+O\left(\epsilon^{2n-2}\right)\right),
\end{align*}
where we used the expansion
\begin{align*}
(1+x)^{a}=1+ax+\frac{a(a-1)}{2!}x^{2}+\frac{a(a-1)(a-2)}{3!}x^{3}+\cdots.
\end{align*}
Hence, we may get the conclusion in each case.
\begin{itemize}
	\item[{\rm (i)}]  When $n+2k \geq 7$,
	\begin{align*}
		&\xi_{(k)}\left(\mathbb {B}^{n}_{1} \setminus \overline{\mathbb {B}^{n}_{\epsilon}}\right) \left|  \partial \left( \mathbb {B}^{n}_{1}  \setminus \overline{\mathbb {B}^{n}_{\epsilon}} \right)\right| ^{\frac{3}{n-1}} \\
		=& \left|  \partial \mathbb {B}^{n}_{1} \right|^{\frac{3}{n-1}} \left\{k^2(n+2k)
		-k^2 (n+2k)\left[k(n+2k)(n+k-2)^{-1}\right.\right.\\
		&\left.\left.~+2\right]\epsilon^{n+2k-2}
		+O\left(\epsilon^{n+2k-1}\right)\right\}
		\cdot
		 \left(1+\frac{3}{n-1}\epsilon^{n-1}+O\left(\epsilon^{2n-2}\right)\right) \\
		 =&\left|  \partial \mathbb {B}^{n}_{1} \right|^{\frac{3}{n-1}} \left(k^2 (n+2k)+\frac{3 k^2 (n+2k)}{n-1}\epsilon^{n-1}
		 +O\left(\epsilon^{n}\right)\right) \\
		 >&k^2 (n+2k) \left|  \partial \mathbb {B}^{n}_{1} \right|^{\frac{3}{n-1}}\\
		 =& \xi_{(k)}\left(\mathbb {B}^{n}_{1} \right) \left|  \partial \mathbb {B}^{n}_{1} \right|^{\frac{3}{n-1}}.
	\end{align*}
	\item[{\rm (ii)}]  When $n+2k =6$,
	\begin{align*}
		&\xi_{(k)}\left(\mathbb {B}^{n}_{1} \setminus \overline{\mathbb {B}^{n}_{\epsilon}}\right) \left|  \partial \left( \mathbb {B}^{n}_{1}  \setminus \overline{\mathbb {B}^{n}_{\epsilon}} \right)\right| ^{\frac{3}{n-1}} \\
		=& \left\{
		k^{2}(n+2k)
		-k^{2}(n+2k)\left[\frac{k(n+2k)}{n+k-2}+\frac{1}{2}(n+2k-3)^{3}\right.\right.\\
		&\left.\left.~+\frac{1}{2}(n+2k+1)\right]\epsilon^{n+2k-2} +O\left(\epsilon^{n+2k-1}\right) \right\} \\
		&\times
		\left(1+\frac{3}{n-1}\epsilon^{n-1}+O\left(\epsilon^{2n-2}\right)\right)  \left|  \partial \mathbb {B}^{n}_{1} \right|^{\frac{3}{n-1}}\\
=&\left|  \partial \mathbb {B}^{n}_{1} \right|^{\frac{3}{n-1}} \cdot  \left(k^2 (n+2k)+\frac{3k^2(n+2k)}{n-1}\epsilon^{n-1}+O\left(\epsilon^{n}\right)\right) \\
		>&k^2 (n+2k) \left|  \partial \mathbb {B}^{n}_{1} \right|^{\frac{3}{n-1}}\\
		=& \xi_{(k)}\left(\mathbb {B}^{n}_{1} \right) \left|  \partial \mathbb {B}^{n}_{1} \right|^{\frac{3}{n-1}}.
	\end{align*}
	\item[{\rm (iii)}]  When $n+2k = 5$,
	\begin{align*}
		&\xi_{(k)}\left(\mathbb {B}^{n}_{1} \setminus \overline{\mathbb {B}^{n}_{\epsilon}}\right) \left|  \partial \left( \mathbb {B}^{n}_{1}  \setminus \overline{\mathbb {B}^{n}_{\epsilon}} \right)\right| ^{\frac{3}{n-1}} \\
		=& \left|  \partial \mathbb {B}^{n}_{1} \right|^{\frac{3}{n-1}} \left\{
		k^{2}(n+2k)
		-k^{2}(n+2k)\left[\frac{k(n+2k)}{n+k-2}+2\right.\right.\\
		&\left.\left.~+\frac{1}{64}(n+2k-2)^{2} (n+2k-4)^{4} \right]\epsilon^{n+2k-2} +O\left(\epsilon^{n+2k-1}\right) \right\} \\
		&\times
		\left(1+\frac{3}{n-1}\epsilon^{n-1}+O\left(\epsilon^{2n-2}\right)\right) \\
		=&\left|  \partial \mathbb {B}^{n}_{1} \right|^{\frac{3}{n-1}} \cdot  \left(k^2 (n+2k)+\frac{3k^2(n+2k)}{n-1}\epsilon^{n-1}+O\left(\epsilon^{n}\right)\right) \\
		>&k^2 (n+2k) \left|  \partial \mathbb {B}^{n}_{1} \right|^{\frac{3}{n-1}}\\
		=& \xi_{(k)}\left(\mathbb {B}^{n}_{1} \right) \left|  \partial \mathbb {B}^{n}_{1} \right|^{\frac{3}{n-1}}.
	\end{align*}
	\item[{\rm (iv)}]  When $n=2,~k=1$,
	\begin{align*}
		&\xi_{(1)}\left(\mathbb {B}^{2}_{1} \setminus \overline{\mathbb {B}^{2}_{\epsilon}}\right) \left|  \partial \left( \mathbb {B}^{2}_{1}  \setminus \overline{\mathbb {B}^{2}_{\epsilon}} \right)\right|^{3}  \\
		=& \left|  \partial \mathbb {B}^{2}_{1} \right|^{3} \left( 4-24\epsilon^{2}+O\left(\epsilon^{2} \log^{-1}\epsilon\right)\right)\cdot
		\left(1+3\epsilon+O\left(\epsilon^{2}\right)\right)  \\
		=&\left|  \partial \mathbb {B}^{2}_{1} \right|^{3} \left(4+12\epsilon+O\left(\epsilon^{2}\right)\right) \\
		>&4 \left|  \partial \mathbb {B}^{2}_{1} \right|^{3}\\
		=& \xi_{(1)}\left(\mathbb {B}^{2}_{1} \right) \left|  \partial \mathbb {B}^{2}_{1} \right|^{3}.
	\end{align*}
\end{itemize}
So the proof is complete.
\end{proof}

\subsection{The spectrum for cylinders}
Next we consider the problem for a cylinder.
\begin{proof}[Proof of Theorem~\ref{thm2}]
The Steklov spectrum of the cylinder $C_{L}$ can be computed by using separation of variables.
	For a cylinder, we can write the eigenfunction as the product of a radial function of the axial variable $s \in [-L,L]$ and a basis function $\beta_{k}$ on $M$.
	
	Let the solution $u(s,x)$ of separate variables to the problem $\eqref{problem1}$  be
	\begin{align*}
		u(s,x)=T(s) \beta(x), \quad s \in [-L,L],~x \in M.
	\end{align*}
We know that the cylinder $C_{L}$ is equipped with the metric $g_{C_{L}}=ds^2+g_{M}$. So the Laplace operator $\Delta_{C_{L}}$ on $C_{L}$ can be expressed as
\begin{align*}
	\Delta_{C_{L}}=\partial^{2}_{s} + \Delta_{M}.
\end{align*}
Thus,
	\begin{align*}
		\Delta_{C_{L}} u
		=&\Delta_{C_{L}} \left(T(s) \beta(x)\right)\\
		=&T^{(2)}(s) \beta(x)+T(s) \Delta_{M} \beta(x)\\
		=&T^{(2)}(s) \beta(x)-c_{k}^{2} T(s)  \beta(x), \\
		\Delta^{2}_{C_{L}} u
		=&\Delta_{C_{L}} \left(T^{(2)}(s) \beta(x)-c_{k}^{2} T(s) \beta(x)\right)\\
		=&T^{(4)}(s) \beta(x)+T^{(2)}(s)  \Delta_{M} \beta(x)-c_{k}^{2} T^{(2)}(s) \beta(x)-c_{k}^{2} T(s) \Delta_{M} \beta(x)\\
		=&T^{(4)}(s) \beta(x)-2 c_{k}^{2} T^{(2)}(s) \beta(x)+c_{k}^{4} T(s) \beta(x)\\
		=&0,
	\end{align*}
where we used $\eqref{DSigma1}$ for the term $\Delta_M \beta(x)$. The eigenvalue problem then reduces to an ordinary differential equation for the axial part
	\begin{align*}
		T^{(4)}(s)-2 c_{k}^{2} T^{(2)}(s) +c_{k}^{4} T(s) =0.
	\end{align*}
	The solutions to the above ordinary differential equation will give the radial part of the Steklov eigenfunctions for the cylinder.
	\begin{itemize}
		\item[{\rm (i)}]  When $k=0$,
		\begin{align*}
			T(s)=d_{1} +s d_2 +s^2 d_{3} +s^3 d_{4}.
		\end{align*}
		\item[{\rm (ii)}]  When $k \geq 1$,
			\begin{align*}
			T(s)=e^{c_{k}s} d_{1} +s e^{c_{k}s} d_2  +e^{-c_{k}s} d_{3} +s e^{-c_{k}s} d_{4} .
		\end{align*}
	\end{itemize}

	The Steklov eigenvalues for the cylinder $C_{L}$ can be found by imposing the boundary conditions, which will lead to a system of linear equations. For the boundary conditions in $\eqref{problem1}$, we choose the outward unit normal vector $\nu=\partial/\partial s$ and $\nu=-\partial/\partial s$ on $\{L\}\times M$ and $\{-L\}\times M$, respectively. As in the proof of Theorem $\ref{thm1}$, we proceed as follows.
	\begin{itemize}
		\item[{\rm (i)}]  When $k=0$, we get
\begin{align*}
T'(s)&=d_2+2sd_3+3s^2d_4,\\
T^{(2)}(s)&=2d_3+6sd_4,\\
\Delta_{C_L}u&=T^{(2)}(s)\beta(x)=(2d_3+6sd_4)\beta(x),
\end{align*}
and so
\begin{align*}
			\left\{\begin{array}{l}
				\begin{aligned}
					&d_{2}+2Ld_{3}+3L^2 d_{4}=0, \\
					&d_{2}-2Ld_{3}+3L^2 d_{4}=0, \\
					&\xi d_{1}+\xi L d_{2}+\xi L^2 d_{3}+\left(\xi L^3+6\right)d_{4}=0,\\
					&\xi d_{1}-\xi L d_{2}+\xi L^2 d_{3}-\left(\xi L^3+6\right)d_{4}=0.
				\end{aligned}
			\end{array}\right.
\end{align*}
	There exists a non-zero solution of the above system of linear equations if and only if
	\begin{align*}
		L^{4}  \xi^2-3 L \xi=0.
	\end{align*}
   So we get
   \begin{align*}
   	\xi^{(0;1)}=0, \quad
   	\xi^{(0;2)}=\frac{3}{L^3}.
   \end{align*}
		\item[{\rm (ii)}]  When $k \geq 1$, we get
\begin{align*}
T'(s)&=c_ke^{c_{k}s} d_{1} +(1+c_ks)e^{c_{k}s} d_2  -c_ke^{-c_{k}s} d_{3} +(1-c_ks) e^{-c_{k}s} d_{4},\\
T^{(2)}(s)&=c_k^2e^{c_{k}s} d_{1} +(2c_k+c_k^2s)e^{c_{k}s} d_2  +c_k^2e^{-c_{k}s} d_{3} +(-2c_k+c_k^2s) e^{-c_{k}s} d_{4},\\
\Delta_{C_L}u&=(T^{(2)}(s)-c_k^2T(s))\beta(x)\nonumber \\
&=(2c_ke^{c_{k}s} d_2 -2c_k e^{-c_{k}s} d_{4})\beta(x),
\end{align*}
and so
\begin{align*}
		\left\{\begin{array}{l}
			\begin{aligned}
				&c_{k} e^{c_{k} L}d_{1}+\left(e^{c_{k} L}+c_{k} L e^{c_{k} L}\right)d_{2}-c_{k} e^{-c_{k} L}d_{3}+\left(e^{-c_{k} L}-c_{k} L e^{-c_{k} L}\right)d_{4}=0, \\
				&c_{k} e^{-c_{k} L}d_{1}+\left(e^{-c_{k} L}-c_{k} L e^{-c_{k} L}\right)d_{2}-c_{k} e^{c_{k} L}d_{3}+\left(e^{c_{k} L}+c_{k} L e^{c_{k} L}\right)d_{4}=0, \\
				&\xi e^{c_{k} L}d_{1}+\left(2c_{k}^{2} e^{c_{k}L}+\xi L e^{c_{k}L}\right) d_{2}+\xi e^{-c_{k}L} d_{3}+\left(2 c_{k}^{2} e^{-c_{k}L}+\xi L e^{-c_{k}L}\right)d_{4}=0,\\
				&\xi e^{-c_{k} L}d_{1}-\left(2c_{k}^{2} e^{-c_{k}L}+\xi L e^{-c_{k}L}\right) d_{2}+\xi e^{c_{k}L} d_{3}-\left(2c_{k}^{2} e^{c_{k}L}+\xi L e^{c_{k}L}\right)d_{4}=0.
			\end{aligned}
		\end{array}\right.
	\end{align*}
The above system of linear equations has a non-zero solution if and only if $\xi$ satisfies
\begin{align*}
	&\left[16 c_{k}^{2} L^{2} e^{4 c_{k} L}-\left(e^{4 c_{k} L}-1\right)^2\right] \xi^{2}
	+\left[4 c_{k}^{3}  \left(e^{8 c_{k} L}-1\right)+32 c_{k}^{4} L e^{4c_{k} L}\right] \xi\\
	&\quad -4c_{k}^{6}\left(e^{4c_{k} L}-1\right)^2
	=0.
\end{align*}
Thus, the eigenvalues $\xi^{(k)}$ are given by
\begin{align*}
	\xi^{(k;1)}=\frac{2c_{k}^{3} \left(e^{2Lc_{k}}+1\right)^2}{e^{4Lc_{k}}-4Lc_{k}e^{2Lc_{k}}-1}, \quad
	\xi^{(k;2)}=\frac{2c_{k}^{3} \left(e^{2Lc_{k}}-1\right)^2}{e^{4Lc_{k}}+4L c_{k}e^{2Lc_{k}}-1}.
\end{align*}
	\end{itemize}

Then the corresponding eigenfunctions can be given as in the theorem.
\end{proof}

\section{Results for the fourth-order Steklov eigenvalue problem of the type two}


\subsection{The asymptotic expansion of the spectrum on annular domains}
In this section, we shall prove  Theorem~\ref{thm3}.

\begin{proof}[Proof of Theorem~\ref{thm3}]
Again we use the method of separating variables and we assume that the eigenfunction $u$ is of the form $u(r,\phi)=\alpha(r)\beta(\phi)$. For the annular domain $\mathbb{B}^{n}_{1} \setminus \overline{\mathbb{B}^{n}_{\epsilon}}$, we still choose the outward unit normal vector $\nu=\partial/\partial r$ and $\nu=-\pt/\pt r$ on $\partial \mathbb{B}^{n}_{1}$ and $\partial \mathbb{B}^{n}_{\epsilon}$, respectively. Therefore, the boundary conditions of the problem \eqref{problem2} are described as follows:
	\begin{equation}
		\begin{aligned}
			u|_{r=1}=&u|_{r=\epsilon }=0,\\
			\left (\Delta u-\eta \frac{\partial u }{\partial r } \right )\bigg|_{r=1}=&\left (\Delta u+\eta \frac{\partial u }{\partial r } \right )\bigg|_{r=\epsilon }=0.
		\end{aligned}
		\label{boundary2}
	\end{equation}
Next we discuss the radial function in different cases.

\textbf{A. The case $n+2k \geq 3$ and $n+2k\neq 4$.} By combining Proposition $\ref{prop1}$ and the boundary conditions \eqref{boundary2}, we obtain a system of linear equations with $a,b,c,d$ as the independent variables:
\begin{equation*}
\left\{\begin{array}{l}
	\begin{aligned}
		&a+b +c+d=0, \\
		&\epsilon^{k}a+\epsilon^{2-n-k} b+\epsilon^{k+2}c+\epsilon^{4-n-k} d=0, \\
		&-k \eta a+\left(n+k-2\right)\eta b+\left[\left( 2n+4k \right) -\left( k+2 \right)\eta \right]c\\
		&~+\left[ \left(8-2n-4k \right) +\left(n+k-4 \right )\eta \right ]d=0,\\
		&k \epsilon^{k-1}\eta a+\left( 2-n-k  \right )\epsilon^{1-n-k}\eta b+\left[ \left( 2n+4k \right)\epsilon^{k} +\left( k+2 \right)\epsilon^{k+1}\eta \right]c\\
		&~+\left[ \left(8-2n-4k \right)\epsilon^{2-n-k} -\left(n+k-4\right)\epsilon^{3-n-k}\eta \right]d=0.		
	\end{aligned}
\label{eta equ}
\end{array}\right.
\end{equation*}
To ensure that the above system of linear equations has a non-zero solution, we get
\begin{align*}
	0&=
	\begin{vmatrix}
		1 & 1 & 1 & 1\\
		\epsilon ^{k} & \epsilon^{2-n-k} & \epsilon^{k+2} & \epsilon^{4-n-k}\\
		-k\eta  & \left(n+k-2\right)\eta  &  2n+4k-\left(k+2\right)\eta  &  b_{34}\\
		k\epsilon^{k-1}\eta &  \left(2-n-k\right)\epsilon^{1-n-k} \eta  & \begin{aligned} &\left(2n+4k\right)\epsilon^{k}\\
			&+\left(k+2\right)\epsilon ^{k+1} \eta
		\end{aligned} & b_{44}
	\end{vmatrix}\\
	&=\begin{vmatrix}
		\epsilon^{2-n-k}-\epsilon^{k} & \epsilon^{k+2}-\epsilon^{k} & \epsilon^{4-n-k}-\epsilon^{k}\\
		(n+2k-2)\eta  & 2n+4k-2\eta  &   b_{34}+k \eta  \\
		\begin{aligned}
			&-k\epsilon ^{k-1} \eta\\
			+&\left(2-n-k\right)\epsilon ^{1-n-k} \eta
		\end{aligned}
		 &\begin{aligned} &\left(2n+4k\right)\epsilon^{k}\\
			+&\left[(k+2) \epsilon ^{k+1}-k \epsilon ^{k-1}\right] \eta
		\end{aligned} & b_{44}-k\epsilon^{k-1} \eta
	\end{vmatrix}\\
	&=A\eta^2+B\eta+C=:g(\eta),
\end{align*}
where
\begin{align*}
b_{34}:=8-2n-4k+\left(n+k-4\right)\eta,
\end{align*}
and
\begin{align*}
b_{44}:=\left(8-2n-4k\right)\epsilon^{2-n-k}-\left(n+k-4\right)\epsilon ^{3-n-k} \eta.
\end{align*}
We can find $A,B,C$ as follows:
\begin{align*}
	C&=g(0)\\
	&=\begin{vmatrix}
		\epsilon^{2-n-k}-\epsilon^{k} & \epsilon^{k+2}-\epsilon^{k} & \epsilon^{4-n-k}-\epsilon^{k}\\
		0  & 2n+4k  &   8-2n-4k  \\
		0
		 &\left(2n+4k\right)\epsilon^{k}
		&\left(8-2n-4k\right)\epsilon^{2-n-k}
	\end{vmatrix}\\
	&=-4(n+2k)(n+2k-4)\left( \epsilon^{2k}
	-2\epsilon^{2-n}+\epsilon^{4-2n-2k} \right),
\end{align*}
\begin{align*}
	B&=g'(0)\\
	&=\begin{vmatrix}
		\epsilon^{2-n-k}-\epsilon^{k} & \epsilon^{k+2}-\epsilon^{k} & \epsilon^{4-n-k}-\epsilon^{k}\\
		n+2k-2 & -2  &  n+2k-4  \\
		0
		&\left(2n+4k\right)\epsilon^{k}
		&\left(8-2n-4k\right)\epsilon^{2-n-k}
	\end{vmatrix}\\
	&~+\begin{vmatrix}
		\epsilon^{2-n-k}-\epsilon^{k} & \epsilon^{k+2}-\epsilon^{k} & \epsilon^{4-n-k}-\epsilon^{k}\\
		0  & 2n+4k  &   8-2n-4k  \\
		\begin{aligned}
			&-k\epsilon ^{k-1}\\
			+&\left(2-n-k\right)\epsilon ^{1-n-k}
		\end{aligned}
		&(k+2) \epsilon ^{k+1}-k \epsilon ^{k-1}
		 & \begin{aligned}
		 	&-k\epsilon ^{k-1}\\
		 	-&\left(n+k-4\right)\epsilon ^{3-n-k}
		 \end{aligned}
	\end{vmatrix}\\
	&=-4(n+2k) \epsilon^{2k}
	-4(n+2k-4)\epsilon^{2k+1}
	-4(n+2k-2)^2 \epsilon^{1-n} \\
	&~-4(n+2k)(n+2k-4)\epsilon^{2-n}
	+4(n+2k)(n+2k-4)\epsilon^{3-n} \\
	&~+4(n+2k-2)^2 \epsilon^{4-n}
	+4(n+2k-4)\epsilon^{4-2n-2k}
	+4(n+2k)\epsilon^{5-2n-2k},
\end{align*}
and
\begin{align*}
	A&=g''(0)/2\\
	&=\begin{vmatrix}
		\epsilon^{2-n-k}-\epsilon^{k} & \epsilon^{k+2}-\epsilon^{k} & \epsilon^{4-n-k}-\epsilon^{k}\\
		n+2k-2 & -2  &  n+2k-4  \\
		\begin{aligned}
			&-k\epsilon ^{k-1}\\
			+&\left(2-n-k\right)\epsilon ^{1-n-k}
		\end{aligned}
		&(k+2) \epsilon ^{k+1}-k \epsilon ^{k-1}
		& \begin{aligned}
			&-k\epsilon ^{k-1}\\
			-&\left(n+k-4\right)\epsilon ^{3-n-k}
		\end{aligned}
	\end{vmatrix}\\
	&=-4\epsilon^{2k+1}
	-4\epsilon^{5-2n-2k}
	+(n+2k-2)^2\epsilon^{5-n}\\
	&~-2(n+2k)(n+2k-4)\epsilon^{3-n}
	+(n+2k-2)^2\epsilon^{1-n}.
\end{align*}

Multiplying $A,B,C$ by
  $-\epsilon^{2n+2k-4}$, and denoting the resulting quantities by the same symbols, we get
  \begin{align*}
	A=&
	4\epsilon
	-(n+2k-2)^2\epsilon^{n+2k-3}\\
	&+2(n+2k)(n+2k-4)\epsilon^{n+2k-1}
	-(n+2k-2)^2\epsilon^{n+2k+1}
	+4\epsilon^{2n+4k-3},\\
	B=&
	-4(n+2k-4)
	-4(n+2k)\epsilon
	+4(n+2k-2)^2 \epsilon^{n+2k-3} \\
	&~+4(n+2k)(n+2k-4)\epsilon^{n+2k-2}
	-4(n+2k)(n+2k-4)\epsilon^{n+2k-1} \\
	&~-4(n+2k-2)^2 \epsilon^{n+2k}
	+4(n+2k) \epsilon^{2n+4k-4}
	+4(n+2k-4)\epsilon^{2n+4k-3},\\
	C=&
	4(n+2k)(n+2k-4)\left( 1
	-2\epsilon^{n+2k-2}+\epsilon^{2n+4k-4}\right).
\end{align*}

Now, we first assume $n+2k \geq 5$. Then we get
\begin{align*}
	&B^2-4AC\\
=&P(n,k,\epsilon)^2+D(n,k)\epsilon^{n+2k-3}+E(n,k)\epsilon^{n+2k-2}+O\left(\epsilon^{n+2k-1}\right)\\
=&P(n,k,\epsilon)^2\left(1+\frac{D(n,k)}{P(n,k,\epsilon)^2}\epsilon^{n+2k-3}+\frac{E(n,k)}{P(n,k,\epsilon)^2}\epsilon^{n+2k-2}+O\left(\epsilon^{n+2k-1}\right)\right),
\end{align*}
where
\begin{align*}
P(n,k,\epsilon)&:=4(n+2k-4)-4(n+2k)\epsilon,\\
D(n,k)&:=4^2 (n+2k-4)(n+2k-2)^3,\\
E(n,k)&:=-2\cdot 4^2 (n+2k)\left[(n+2k-4)^2+(n+2k-2)^2\right].
\end{align*}
So we have
\begin{align*}
	&\sqrt{B^2-4AC}\\
=&P(n,k,\epsilon)\left(1+\frac{D(n,k)}{P(n,k,\epsilon)^2}\epsilon^{n+2k-3}+\frac{E(n,k)}{P(n,k,\epsilon)^2}\epsilon^{n+2k-2}+O\left(\epsilon^{n+2k-1}\right)\right)^{1/2}\\
=&P(n,k,\epsilon)\left(1+\frac{D(n,k)}{2P(n,k,\epsilon)^2}\epsilon^{n+2k-3}+\frac{E(n,k)}{2P(n,k,\epsilon)^2}\epsilon^{n+2k-2}+O\left(\epsilon^{n+2k-1}\right)\right)\\
=&P(n,k,\epsilon)+\frac{D(n,k)}{2P(n,k,\epsilon)}\epsilon^{n+2k-3}+\frac{E(n,k)}{2P(n,k,\epsilon)}\epsilon^{n+2k-2}+O\left(\epsilon^{n+2k-1}\right)\\
=&P(n,k,\epsilon)+\frac{D(n,k)}{2P(n,k,0)}\epsilon^{n+2k-3}\\
&+\left(\frac{E(n,k)}{2P(n,k,0)}+\frac{D(n,k)}{2}\frac{n+2k}{4(n+2k-4)^2}\right)\epsilon^{n+2k-2}+O\left(\epsilon^{n+2k-1}\right),
\end{align*}
where we used
\begin{align*}
\frac{1}{P(n,k,\epsilon)}
	=&\frac{1}{4(n+2k-4)} \cdot \left(1-\frac{n+2k}{n+2k-4}\epsilon\right)^{-1}\\
	=&\frac{1}{4(n+2k-4)} \cdot \left[
	1
	+\frac{n+2k}{n+2k-4}\epsilon
	+ O\left(\epsilon^{2}\right)
	\right].\\
\end{align*}
Simplifying the above expression, we get
\begin{align*}
	&\sqrt{B^2-4AC}\\
=&4(n+2k-4)-4(n+2k)\epsilon
	+2(n+2k-2)^3\epsilon^{n+2k-3}\\
	&+2(n+2k) \left[(n+2k-2)^2-2(n+2k-4)\right]\epsilon^{n+2k-2}
	+O\left(\epsilon^{n+2k-1}\right).
\end{align*}
Besides we deduce
\begin{equation*}
\begin{aligned}
		\frac{1}{A}
		=& \frac{1}{4\epsilon \left[1-\dfrac{(n+2k-2)^2}{4}\epsilon^{n+2k-4}+O\left(\epsilon^{n+2k-2}\right)\right]}\\
		=& \frac{1}{4\epsilon}\left[1+\frac{(n+2k-2)^2}{4}\epsilon^{n+2k-4}+\frac{(n+2k-2)^4}{16}\epsilon^{2n+4k-8}+O\left(\epsilon^{n+2k-2}\right)\right]\\
		=& \frac{1}{4}\left[\epsilon^{-1}+\frac{(n+2k-2)^2}{4}\epsilon^{n+2k-5}+\frac{(n+2k-2)^4}{16}\epsilon^{2n+4k-9}+O\left(\epsilon^{n+2k-3}\right)\right].
\end{aligned}
\label{21a}
\end{equation*}
Thus, we may calculate that the quadratic equation for $\eta$ has roots
\begin{align*}
	\eta&=\frac{1}{8} \left[\epsilon^{-1}+\frac{(n+2k-2)^2}{4}\epsilon^{n+2k-5}+\frac{(n+2k-2)^4}{16}\epsilon^{2n+4k-9}+O\left(\epsilon^{n+2k-3}\right)\right]\\
	&\cdot \left\{
	4(n+2k-4)
	+4(n+2k)\epsilon
	-4(n+2k-2)^2 \epsilon^{n+2k-3} \right.\\
	&\left.~-4(n+2k)(n+2k-4)\epsilon^{n+2k-2}
	+4(n+2k)(n+2k-4)\epsilon^{n+2k-1} \right.\\
	&\left.~+4(n+2k-2)^2 \epsilon^{n+2k}
	-4(n+2k) \epsilon^{2n+4k-4}
	-4(n+2k-4)\epsilon^{2n+4k-3}\right.\\
	&\left.~ \pm \left[
	4(n+2k-4)-4(n+2k)\epsilon
	+2(n+2k-2)^3\epsilon^{n+2k-3}\right.\right.\\
	&\left.\left.~+2(n+2k) \left[(n+2k-2)^2-2(n+2k-4)\right]\epsilon^{n+2k-2}
	+O\left(\epsilon^{n+2k-1}\right)
	\right]
	\right\}.
\end{align*}

To proceed, we need to compare the orders of the terms $\epsilon^{n+2k-5}$ and $\epsilon^{2n+4k-9}$ in the first line of the above expression for $\eta$. Thus, we have two cases.
\begin{itemize}
	\item[{\rm (i)}]  The case  $n+2k \geq 6$. The roots of the equation are
	\begin{align*}
		\eta^{(1)}_{(k)}
		&= \frac{1}{8}
		\left[
		\epsilon^{-1}+\frac{(n+2k-2)^2}{4}\epsilon^{n+2k-5}+O\left(\epsilon^{n+2k-3}\right)\right]
		\cdot\left\{
		8(n+2k)\epsilon\right.\\
		&\left.~-2(n+2k)(n+2k-2)^2\epsilon^{n+2k-3}
		-2(n+2k)(n+2k-2)^2 \epsilon^{n+2k-2}\right.\\
		&\left.~+O\left(\epsilon^{n+2k-1}\right)
		\right\}\\	
		&=n+2k
		-\frac{1}{4}(n+2k)(n+2k-2)^2\epsilon^{n+2k-3}
		+O\left(\epsilon^{n+2k-2}\right),\\
		\eta^{(2)}_{(k)} &=\infty.
	\end{align*}
	\item[{\rm (ii)}]  The case  $n+2k=5$. The roots of the equation are
\begin{align*}
		\eta^{(1)}_{(k)}&= \frac{1}{8} \left[\epsilon^{-1}+\frac{(n+2k-2)^2}{4}\epsilon^{n+2k-5}+\frac{(n+2k-2)^4}{16}\epsilon^{n+2k-4}\right.\\
		&\left.+O\left(\epsilon^{n+2k-3}\right)\right]
       \cdot\left\{8(n+2k)\epsilon
		-2(n+2k)(n+2k-2)^2\epsilon^{n+2k-3}
		\right.\\
&-2(n+2k)(n+2k-2)^2 \epsilon^{n+2k-2}\left. +O\left(\epsilon^{n+2k-1}\right)\right\}\\	
		&=n+2k-\frac{1}{8} \left[2(n+2k)(n+2k-2)^2+\frac{1}{2}(n+2k)(n+2k-2)^4 \right.\\
 &\left. -\frac{1}{2}(n+2k)(n+2k-2)^4\right]\epsilon^{n+2k-3}
		+O\left(\epsilon^{n+2k-2}\right)\\
		&=n+2k-\frac{1}{4}(n+2k)(n+2k-2)^2\epsilon^{n+2k-3}+O\left(\epsilon^{n+2k-2}\right),\\
		\eta^{(2)}_{(k)} &=\infty.
\end{align*}		
\end{itemize}

Last, we consider the remaining case $n+2k = 3~(i.e., ~n=3,~k=0)$. In this case, the determinant of the coefficient matrix equal to zero is equivalent to the following quadratic equation with one variable
\begin{align*}
	\left(\epsilon-1\right)^{4} \eta^{2}+8\left(\epsilon-1\right)^{3} \eta+12\left(\epsilon-1\right)^{2} =0.
\end{align*}
Obviously, the roots of the equation are
		\begin{align*}
			\eta^{(1)}_{(0)}
			= \frac{2}{1-\epsilon}
			=2+2\epsilon +O\left(\epsilon^{2}\right),\quad
			\eta^{(2)}_{(0)}
			= \frac{6}{1-\epsilon}
			=6+6\epsilon +O\left(\epsilon^{2}\right).
		\end{align*}

\textbf{B. The case $n+2k=2$, i.e., $n=2$,  $k=0$.}
In this case we have
\begin{align*}
u(r,\phi)&=\left(a+b\log r+cr^2+dr^2\log r\right)\beta(\phi),\\
\frac{\pt u}{\pt r}&=\left(br^{-1}+2cr+d(2r\log r+r)\right)\beta(\phi),\\
\Delta u&=\left(4c+d(4\log r+4)\right)\beta(\phi).
\end{align*}
By the boundary conditions,
we can obtain the following system of linear equations with $a,b,c,d$ as the independent variables:
$$
\left\{\begin{array}{l}
	\begin{aligned}
		&a+c=0, \\
		&a+\log\epsilon b+\epsilon^{2}c+\epsilon^{2} \log\epsilon d=0,\\
		&-\eta b+\left(4-2\eta\right)c+\left(4-\eta\right)d=0,\\
		&\eta \epsilon^{-1} b+\left(4+2\eta \epsilon\right)c+\left(4\log\epsilon+4+2\eta \epsilon \log\epsilon +\eta \epsilon\right)d=0.
	\end{aligned}
\end{array}\right.
$$
To ensure that the above system of linear equations has non-zero solutions, the eigenvalue $\eta$ must satisfy
\begin{equation}
	\begin{aligned}
	f(\eta):=&A\eta^2 +B\eta +C\\
	=&\left(1
		-4\epsilon^{2} \log^{2}\epsilon
		-2\epsilon^{2}
		+\epsilon^{4} \right) \eta^{2}+
		\left(-4
		-8\epsilon \log^{2}\epsilon
		-8\epsilon\log\epsilon
		-4\epsilon
		 \right.\\
		&\left.~+8\epsilon^2 \log^{2}\epsilon
		-8\epsilon^2 \log\epsilon
		+4\epsilon^{2}
		 +4\epsilon^{3}
		 \right) \eta
		+16\epsilon \log^{2}\epsilon\\
	=&0.
	\end{aligned}
\label{n=2,k=0}
\end{equation}
Solving the above equation we get
\begin{align*}
\eta=&\frac{-B \pm \sqrt{B^{2}-4AC}}{2A}\\
=&\frac{1}{2} \left(1+O\left(\epsilon^{2}\log^{2}\epsilon \right)\right)\cdot \left[4+8\epsilon\log^{2}\epsilon+O\left(\epsilon \log\epsilon\right)
\pm
\left(4+O\left(\epsilon \log\epsilon\right)\right)\right].
\end{align*}
So the roots of the equation $\eqref{n=2,k=0}$ are given by
\begin{align*}
	\eta^{(1)}_{(0)}
	= 4\epsilon \log^{2}\epsilon+O\left(\epsilon \log\epsilon\right),\quad
	\eta^{(2)}_{(0)} =4+4\epsilon \log^{2}\epsilon +O\left(\epsilon \log \epsilon  \right).
\end{align*}

\textbf{C. The case $n+2k=4$, i.e., $n=2$, $k= 1$ or $n=4$, $k=0$.}
\begin{itemize}
	\item[{\rm (i)}]  When $n=2,~k=1$, we have
\begin{align*}
u(r,\phi)&=\left(ar+br^{-1}+cr^3+dr\log r\right)\beta(\phi),\\
\frac{\pt u}{\pt r}&=\left(a-br^{-2}+3cr^2+d(\log r+1)\right)\beta(\phi),\\
\Delta u&=\left(8cr+2dr^{-1}\right)\beta(\phi),
\end{align*}
and so
	$$
	\left\{\begin{array}{l}
		\begin{aligned}
			&a+b+c=0, \\
			&\epsilon a+\epsilon^{-1} b+\epsilon^{3}c+\epsilon \log\epsilon d=0,\\
			&-\eta a+\eta b+\left(8-3\eta\right)c+\left(2-\eta\right)d=0,\\
			&\eta a-\eta \epsilon^{-2} b+\left(8\epsilon+3\eta \epsilon^{2}\right)c+\left(2\epsilon^{-1}+ \eta \log\epsilon +\eta \right)d=0.
		\end{aligned}
	\end{array}\right.
	$$
	\item[{\rm (ii)}]  When $n=4,~k=0$, we have
\begin{align*}
u(r,\phi)&=\left(a+br^{-2}+cr^2+d\log r\right)\beta(\phi),\\
\frac{\pt u}{\pt r}&=\left(-2br^{-3}+2cr+dr^{-1}\right)\beta(\phi),\\
\Delta u&=\left(8c+2dr^{-2}\right)\beta(\phi),
\end{align*}
and so
	$$
	\left\{\begin{array}{l}
		\begin{aligned}
			&a+b+c=0, \\
			&a+\epsilon^{-2} b+\epsilon^{2}c+ \log\epsilon d=0,\\
			&2\eta b+\left(8-2\eta\right)c+\left(2-\eta\right)d=0,\\
			&-2\eta \epsilon^{-3} b+\left(8+2\eta \epsilon\right)c+\left(2 \epsilon^{-2} +\eta \epsilon^{-1}  \right)d=0.
		\end{aligned}
	\end{array}\right.
	$$
\end{itemize}
One can check that the above two systems of linear equations are equivalent to each other. They have non-zero solutions if and only if $\eta$ satisfies
\begin{equation}\label{n=2,k=1}
	\begin{aligned}
		&\left(\epsilon \log\epsilon
		+\epsilon-2\epsilon^{3}
		-\epsilon^{5}\log\epsilon
		+\epsilon^{5} \right) \eta^{2}+
		\left(1
		-4\epsilon\log\epsilon
		-3\epsilon
		-4\epsilon^{2}
		+4\epsilon^{3}\right.\\
		&\left.~-4\epsilon^{4} \log\epsilon
		+3\epsilon^{4}
		-\epsilon^{5}
		\right) \eta
		-4+8\epsilon^{2}-4\epsilon^{4}
		=0.
	\end{aligned}
\end{equation}
%
%
Writing the equation \eqref{n=2,k=1} as $A\eta^2+B\eta+C=0$, we get
\begin{align*}
	B^2-4AC	=& \left(1+4\epsilon\log\epsilon\right)^{2}
	+10\epsilon
	+24\epsilon^{2} \log\epsilon
	+O\left(\epsilon^{2}\right)\\
	=& \left(1+4\epsilon\log\epsilon\right)^{2} \cdot \left[1+\frac{10\epsilon}{\left(1+4\epsilon\log\epsilon\right)^{2} }
	+\frac{24\epsilon^{2} \log\epsilon}{\left(1+4\epsilon\log\epsilon\right)^{2}}
	+O\left(\epsilon^{2}\right)
	\right],
\end{align*}
which implies
\begin{align*}
	&\sqrt{B^2-4AC}\\
	=&\left(1+4\epsilon\log\epsilon\right)
	\cdot \left(1+\frac{1}{2}\left(\frac{10\epsilon}{\left(1+4\epsilon\log\epsilon\right)^{2} }
	+\frac{24\epsilon^{2} \log\epsilon}{\left(1+4\epsilon\log\epsilon\right)^{2}}
	+O\left(\epsilon^{2}\right)
	\right)\right)\\
	=&1
	+4\epsilon \log\epsilon
	+5\epsilon
	-8\epsilon^{2} \log\epsilon
	+O\left(\epsilon^{2}\right).
\end{align*}
Besides, we note
\begin{align*}
   \frac{1}{2A}
   &=\frac{1}{2\epsilon (\log\epsilon+1)}\left(1+O(\epsilon^2 \log^{-1}\epsilon)\right)^{-1}\\
   &=\frac{1}{2\epsilon (\log\epsilon+1)}\left(1+O(\epsilon^2 \log^{-1}\epsilon)\right).
\end{align*}
So we see
\begin{align*}
	\eta&=\frac{1}{2\epsilon (\log\epsilon+1)}\left(1+O(\epsilon^2 \log^{-1}\epsilon)\right)
	\cdot \left[-1
	+4\epsilon\log\epsilon
	+3\epsilon
	+4\epsilon^{2}
	-4\epsilon^{3}\right.\\
	&\left.\quad+4\epsilon^{4} \log\epsilon
	-3\epsilon^{4}
	+\epsilon^{5}
	 \pm
	\left(1
	+4\epsilon \log\epsilon
	+5\epsilon
	-8\epsilon^{2} \log\epsilon
	+O\left(\epsilon^{2}\right)\right)
	\right].
\end{align*}
Thus,
\begin{align*}
	\eta^{(1)}_{(k)}
	&=\frac{1}{2\epsilon (\log\epsilon+1)}\left(1+O(\epsilon^2 \log^{-1}\epsilon)\right) \left(8\epsilon(\log \epsilon+1)-8\epsilon^2\log \epsilon+O(\epsilon^2)\right)\\
	&=4-4\epsilon
	+O\left(\epsilon \log^{-1}\epsilon\right),\\
	\eta^{(2)}_{(k)} &=\infty.
\end{align*}

For the last part of the proof, we first have
\begin{align*}
	\left|  \partial \left( \mathbb {B}^{n}_{1}  \setminus \overline{\mathbb {B}^{n}_{\epsilon}} \right)\right| ^{\frac{1}{n-1}}
	=
	\left|  \partial \mathbb {B}^{n}_{1} \right| ^{\frac{1}{n-1}} \left(1+\frac{1}{n-1}\epsilon^{n-1}+O\left(\epsilon^{2n-2}\right)\right).
\end{align*}
Next we will investigate the relationship between the spectrum of $\mathbb {B}^{n}_{1} \setminus \overline{\mathbb {B}^{n}_{\epsilon}}$ and that of $\mathbb {B}^{n}_{1}$ in each case. And as explained in Remark~\ref{exceptional-notation}, here we will not consider the cases of $n=2$ and $3$.
\begin{itemize}
	\item[{\rm (i)}]  When $n+2k \geq 5$ and $n \geq 4$,
	\begin{align*}
		&\eta_{(k)}\left(\mathbb {B}^{n}_{1} \setminus \overline{\mathbb {B}^{n}_{\epsilon}}\right) \left|  \partial \left( \mathbb {B}^{n}_{1}  \setminus \overline{\mathbb {B}^{n}_{\epsilon}} \right)\right| ^{\frac{1}{n-1}} \\
		=&  \left|  \partial \mathbb {B}^{n}_{1} \right|^{\frac{1}{n-1}} \left\{
		n+2k
		-\frac{1}{4} (n+2k)(n+2k-2)^{2}\epsilon^{n+2k-3}
		+O\left(\epsilon^{n+2k-2}\right)\right\} \\
		&\times
		\left(1+\frac{1}{n-1}\epsilon^{n-1}+O\left(\epsilon^{2n-2}\right)\right).
	\end{align*}
By comparing the orders of the terms $\epsilon^{n+2k-3}$ and $\epsilon^{n-1}$, we get the following three sub-cases.
\begin{itemize}
\item[{\rm (a)}]  When $n \geq 5$ and $k =0$, we have
\begin{align*}
	&\eta_{(0)}\left(\mathbb {B}^{n}_{1} \setminus \overline{\mathbb {B}^{n}_{\epsilon}}\right) \left|  \partial \left( \mathbb {B}^{n}_{1}  \setminus \overline{\mathbb {B}^{n}_{\epsilon}} \right)\right| ^{\frac{1}{n-1}} \\
	=& \left|  \partial \mathbb {B}^{n}_{1} \right|^{\frac{1}{n-1}} \left(n-\frac{1}{4} n(n-2)^{2}\epsilon^{n-3}
	+O\left(\epsilon^{n-2}\right)\right)\\
	<&n \left|  \partial \mathbb {B}^{n}_{1} \right|^{\frac{1}{n-1}}\\
	=&\eta_{(0)} \left(\mathbb {B}^{n}_{1} \right) \left|  \partial \mathbb {B}^{n}_{1} \right|^{\frac{1}{n-1}}.
\end{align*}
\item[{\rm (b)}]  When $n \geq 4$ and $k =1$, we deduce
\begin{align*}
	&\eta_{(1)}\left(\mathbb {B}^{n}_{1} \setminus \overline{\mathbb {B}^{n}_{\epsilon}}\right) \left|  \partial \left( \mathbb {B}^{n}_{1}  \setminus \overline{\mathbb {B}^{n}_{\epsilon}} \right)\right| ^{\frac{1}{n-1}} \\
	=& \left|  \partial \mathbb {B}^{n}_{1} \right|^{\frac{1}{n-1}}  \left\{n+2 +\left[\frac{n+2}{n-1}-\frac{1}{4} n^{2} (n+2)\right]\epsilon^{n-1}
	+O\left(\epsilon^{n}\right)\right\}\\
	<&\left(n+2\right) \left|  \partial \mathbb {B}^{n}_{1} \right|^{\frac{1}{n-1}}\\
	=&\eta_{(1)} \left(\mathbb {B}^{n}_{1} \right) \left|  \partial \mathbb {B}^{n}_{1} \right|^{\frac{1}{n-1}}.
\end{align*}
\item[{\rm (c)}]  When $n \geq 4$ and $k \geq 2$, we obtain
\begin{align*}
	&\eta_{(k)}\left(\mathbb {B}^{n}_{1} \setminus \overline{\mathbb {B}^{n}_{\epsilon}}\right) \left|  \partial \left( \mathbb {B}^{n}_{1}  \setminus \overline{\mathbb {B}^{n}_{\epsilon}} \right)\right| ^{\frac{1}{n-1}} \\
	=&\left|  \partial \mathbb {B}^{n}_{1} \right|^{\frac{1}{n-1}} \left(n+2k+\frac{n+2k}{n-1} \epsilon^{n-1}+O\left(\epsilon^{n}\right) \right)\\
	>&\left(n+2k\right) \left|  \partial \mathbb {B}^{n}_{1} \right|^{\frac{1}{n-1}}\\
	=&\eta_{(k)} \left(\mathbb {B}^{n}_{1} \right) \left|  \partial \mathbb {B}^{n}_{1} \right|^{\frac{1}{n-1}}.
\end{align*}
\end{itemize}
	\item[{\rm (ii)}]  When $n=4$ and $k=0$,
	\begin{align*}
		&\eta_{(0)}\left(\mathbb {B}^{4}_{1} \setminus \overline{\mathbb {B}^{4}_{\epsilon}}\right) \left|  \partial \left( \mathbb {B}^{4}_{1}  \setminus \overline{\mathbb {B}^{4}_{\epsilon}} \right)\right| ^{\frac{1}{3}} \\
		=& \left|  \partial \mathbb {B}^{4}_{1} \right|^{\frac{1}{3}} \left(4-4\epsilon +O\left(\epsilon \log^{-1}\epsilon \right)
		 \right)
		\times
		\left(1+\frac{1}{3} \epsilon^{3}+O\left(\epsilon^{6}\right)\right) \\
		=&\left|  \partial \mathbb {B}^{4}_{1} \right|^{\frac{1}{3}}  \left(4-4\epsilon+O\left(\epsilon \log^{-1}\epsilon \right)\right) \\
		<&4 \left|  \partial \mathbb {B}^{4}_{1} \right|^{\frac{1}{3}}\\
		=& \eta_{(0)}\left(\mathbb {B}^{4}_{1} \right) \left|  \partial \mathbb {B}^{4}_{1} \right|^{\frac{1}{3}}.
	\end{align*}
\end{itemize}
So the proof is complete.
\end{proof}

\subsection{The spectrum for cylinders}
Next we will give the proof of Theorem~\ref{thm4}.

\begin{proof}[Proof of Theorem~\ref{thm4}]
As in the proof of Theorem $\ref{thm2}$, we can obtain two systems of linear equations depending on the value of $k$.
\begin{itemize}
	\item[{\rm (i)}]  When $k=0$, we get
\begin{align*}
T'(s)&=d_2+2sd_3+3s^2d_4,\\
T^{(2)}(s)&=2d_3+6sd_4,\\
\Delta_{C_L}u&=T^{(2)}(s)\beta(x)=(2d_3+6sd_4)\beta(x),
\end{align*}
and so
	\begin{align*}
		\left\{\begin{array}{l}
			\begin{aligned}
				&d_{1}+Ld_{2}+L^{2}d_{3}+L^{3} d_{4}=0, \\
				&d_{1}-Ld_{2}+L^{2}d_{3}-L^{3} d_{4}=0, \\
				&-\eta d_{2}+ \left(2-2\eta L\right)  d_{3}+\left(6L-3\eta L^2\right)d_{4}=0,\\
				&\eta d_{2}+ \left(2-2\eta L\right) d_{3}+\left(3\eta L^2-6L\right)d_{4}=0.
			\end{aligned}
		\end{array}\right.
	\end{align*}
	There exists a non-zero solution of the above system if and only if
	\begin{align*}
		16L^{4} \eta^2-64L^{3} \eta+48L^{2}=0.
	\end{align*}
	So we get
	\begin{align*}
		\eta^{(0;1)}=\frac{1}{L}, \quad
		\eta^{(0;2)}=\frac{3}{L}.
	\end{align*}
	\item[{\rm (ii)}]  When $k \geq 1$, letting $c_k=\lambda_k(M)^{1/2}$, we get
\begin{align*}
T'(s)&=c_ke^{c_{k}s} d_{1} +(1+c_ks)e^{c_{k}s} d_2  -c_ke^{-c_{k}s} d_{3} +(1-c_ks) e^{-c_{k}s} d_{4},\\
T^{(2)}(s)&=c_k^2e^{c_{k}s} d_{1} +(2c_k+c_k^2s)e^{c_{k}s} d_2  +c_k^2e^{-c_{k}s} d_{3} +(-2c_k+c_k^2s) e^{-c_{k}s} d_{4},\\
\Delta_{C_L}u&=(T^{(2)}(s)-c_k^2T(s))\beta(x)\nonumber \\
&=(2c_ke^{c_{k}s} d_2 -2c_k e^{-c_{k}s} d_{4})\beta(x),
\end{align*}
and so
	\begin{align*}
		\left\{\begin{array}{l}
			\begin{aligned}
				&e^{c_{k} L}d_{1}
				+L e^{c_{k} L}d_{2}
				+e^{-c_{k} L}d_{3}
				+L e^{-c_{k} L}d_{4}=0, \\
				&e^{-c_{k} L}d_{1}
				-L e^{-c_{k} L}d_{2}
				+e^{c_{k} L}d_{3}
				-L e^{c_{k} L}d_{4}=0, \\
				&-\eta c_{k} e^{c_{k} L}d_{1}
				+\left(2c_{k} e^{c_{k} L}-\eta e^{c_{k} L}-\eta c_{k} L e^{c_{k} L}\right) d_{2}\\
				&+\eta c_{k} e^{-c_{k} L} d_{3}
				-\left(2c_{k} e^{-c_{k} L}+\eta e^{-c_{k} L}-\eta c_{k} L e^{-c_{k} L}\right)d_{4}=0,\\
				&\eta c_{k}e^{-c_{k}  L}d_{1}
				+\left(2c_{k} e^{-c_{k} L}+\eta e^{-c_{k} L}-\eta c_{k} L e^{-c_{k} L}\right) d_{2}\\
				&-\eta c_{k} e^{c_{k} L} d_{3}
				-\left(2c_{k} e^{c_{k} L}-\eta e^{c_{k} L}-\eta c_{k} L e^{c_{k} L}\right)d_{4}=0.
			\end{aligned}
		\end{array}\right.
	\end{align*}
	The above system of linear equations has a non-zero solution if and only if $\eta$ satisfies
	\begin{align*}
		&\left[\left(e^{4c_{k} L}-1\right)^2-16c_{k}^{2}L^{2} e^{4c_{k} L}\right] \eta^{2}
		+\left[32 c_{k}^{2} L e^{4c_{k} L}-4c_{k} \left(e^{8c_{k} L}-1\right)\right] \eta\\
		+&4c_{k}^{2} \left(e^{4c_{k}  L}-1\right)^2
		=0.
	\end{align*}
	Thus, the eigenvalues $\eta^{(k)}$ are given by
	\begin{align*}
		\eta^{(k;1)}=\frac{2c_{k} \left(e^{2 L c_{k}}-1\right)^2}{e^{4Lc_{k}}-4Lc_{k} e^{2Lc_{k}}-1}, \quad
		\eta^{(k;2)}=\frac{2c_{k} \left(e^{2 L c_{k}}+1\right)^2}{e^{4Lc_{k}}+4Lc_{k} e^{2Lc_{k}}-1}.
	\end{align*}
\end{itemize}

Then the corresponding eigenfunctions can be given as in the theorem.

\end{proof}

\section{Results for the fourth-order Steklov eigenvalue problem of the type three}
In this section, we derive a sharp upper bound for the first non-zero eigenvalue of the fourth-order Steklov eigenvalue problem~\eqref{problem3} on a star-shaped and mean convex domain in $\mathbb{R}^{n}$.

\begin{proof}[Proof of Theorem~\ref{thm5}] By translating the origin, we may assume that
\begin{align*}
\int_{\pt \Omega} x_i d\sigma=0,\quad i=1,\dots, n.
\end{align*}
Then using $x_i$ ($1\leq i\leq n$) as test functions in the variational characterization~\eqref{variational-characterization} for $\rho_1$ we get
\begin{align*}
\rho_1  \int_{\pt \Omega} x_i^2 d\sigma \leq \tau \int_\Omega dx,\quad i=1,\dots, n,
\end{align*}
which after summation over $i=1,\dots,n$ yields
\begin{align*}
\rho_1 \int_{\pt \Omega} |x|^2 d\sigma \leq n \tau |\Omega|.
\end{align*}
Now we employ the inequality (1.4) in \cite{KW23} to derive
\begin{align*}
|\pt \Omega| \int_{\pt \Omega} |x|^2 d\sigma &\geq \left(\int_{\pt \Omega} |x|d\sigma \right)^2\\
&\geq \left(\frac{n-1}{n}\omega_{n-1}\left(\frac{|\pt \Omega|}{\omega_{n-1}}\right)^{n/(n-1)}+|\Omega|\right)^2,
\end{align*}
where $\omega_{n-1}=|\SS^{n-1}|$, with the equality if and only if $\Omega$ is a Euclidean ball centered at the origin. Using this inequality, we obtain
\begin{align}
\rho_1\leq n\tau |\Omega||\pt \Omega|\left(\frac{n-1}{n}\omega_{n-1}\left(\frac{|\pt \Omega|}{\omega_{n-1}}\right)^{n/(n-1)}+|\Omega|\right)^{-2}.
\end{align}
Furthermore, if the equality holds, we see that $\Omega$ is a Euclidean ball centered at the origin. The proof is complete.
\end{proof}

\begin{rem}
Denote by $p_k$ the $k$th eigenvalue for the second-order Steklov eigenvalue problem. Using the min-max variational principles for $p_k$ and $\rho_k$, we may prove
\begin{align*}
\tau p_k\leq \rho_k,\quad k=0,\dots.
\end{align*}
In fact, we see
\begin{align*}
\rho_k(\Omega)&=\inf_{V_{k}\subset H^2(\Omega),\ \mathrm{dim}(V_{k})=k} \sup_{u \in  V_{k},\ u|_{\pt \Omega}\neq 0}\frac{\int_\Omega |D^2 u|^2+\tau |\nabla u|^2 dx}{\int_{\pt \Omega}u^2 d\sigma}\\
&\geq \tau \inf_{V_{k}\subset H^2(\Omega),\ \mathrm{dim}(V_{k})=k} \sup_{u \in  V_{k},\ u|_{\pt \Omega}\neq 0}\frac{\int_\Omega |\nabla u|^2 dx}{\int_{\pt \Omega}u^2 d\sigma}\\
& \geq \tau \inf_{V_{k}\subset H^1(\Omega),\ \mathrm{dim}(V_{k})=k} \sup_{u \in  V_{k},\ u|_{\pt \Omega}\neq 0}\frac{\int_\Omega |\nabla u|^2 dx}{\int_{\pt \Omega}u^2 d\sigma}\\
&=\tau p_k(\Omega).
\end{align*}
So as a corollary, we get after using the Young inequality
\begin{align*}
p_1&\leq n  |\Omega||\pt \Omega|\left(\frac{n-1}{n}\omega_{n-1}\left(\frac{|\pt \Omega|}{\omega_{n-1}}\right)^{n/(n-1)}+|\Omega|\right)^{-2}\\
&\leq  \frac{n^{1-2/n}\omega_{n-1}^{2/n} |\Omega|^{1-2/n}}{|\pt \Omega|},
\end{align*}
the result in \cite{KW23}.
\end{rem}
\noindent \textbf{Conflict of interest}: The authors have no competing interests to declare that are relevant to the content of this article.



\bibliographystyle{Plain}

\begin{thebibliography}{99}




\bibitem{AG13} Pedro Ricardo Sim\~{a}o~Antunes and Filippo~Gazzola, \emph{Convex shape optimization for the least biharmonic Steklov eigenvalue}, ESAIM Control Optim. Calc. Var. \textbf{19} (2013), no.~2, 385--403.

\bibitem{ABR92} Sheldon~Axler, Paul~Bourdon, and Wade~Ramey, \emph{Harmonic function theory}, Graduate Texts in Mathematics, \textbf{137}. Springer-Verlag, New York, 1992. xii+231 pp.

\bibitem{BLSV23} Rondinelle~Batista, Barnab\'{e}~Lima, Paulo~Sousa, and Bruno~Vieira, \emph{Estimate for the first fourth Steklov eigenvalue of a minimal hypersurface with free boundary}, Pacific J. Math. \textbf{325} (2023), no.~1, 1--10.


\bibitem{Bro01} F.~Brock, \emph{An isoperimetric inequality for eigenvalues of the Stekloff problem}, ZAMM Z. Angew. Math. Mech. \textbf{81} (2001), no.~1, 69--71.

\bibitem{BFG09} D.~Bucur, A.~Ferrero, and F.~Gazzola, \emph{On the first eigenvalue of a fourth order Steklov problem},  Calc. Var. Partial Differential Equations \textbf{35} (2009), no.~1, 103--131.

\bibitem{BFNT} D.~Bucur, V.~Ferone, C.~Nitsch, and C.~Trombetti, \emph{Weinstock inequality in higher dimensions}, J. Differential Geom. \textbf{118} (2021), no.~1, 1--21.

\bibitem{BG11} Dorin~Bucur and Filippo~Gazzola, \emph{The first biharmonic Steklov eigenvalue: positivity preserving and shape optimization}, Milan J. Math. \textbf{79} (2011), no.~1, 247--258.

\bibitem{BP15} D.~Buoso and L.~Provenzano, \emph{A few shape optimization results for a biharmonic Steklov problem}, J. Differential Equations \textbf{259} (2015), no.~5, 1778--1818.



\bibitem{CGGS24} Bruno~Colbois, Alexandre~Girouard, Carolyn~Gordon, and David~Sher, \emph{Some recent developments on the Steklov eigenvalue problem}, Rev. Mat. Complut. \textbf{37} (2024), no.~1, 1--161.

\bibitem{DMWXZ} Feng~Du, Jing~Mao, Qiaoling~Wang, Changyu~Xia, and Yan~Zhao, \emph{Estimates for eigenvalues of the Neumann and Steklov problems}, Adv. Nonlinear Anal. \textbf{12} (2023), no.~1, Paper No. 20220321, 12 pp.



\bibitem{FGW05}  A.~Ferrero, F.~Gazzola, and T.~Weth, \emph{On a  fourth order Steklov eigenvalues problems},  Analysis \textbf{25} (2005), 315--332.

\bibitem{FS19} A.~Fraser and R.~Schoen, \emph{Shape optimization for the Steklov problem in higher dimensions}, Adv. Math. \textbf{348} (2019), 146--162.

\bibitem{GP17} A.~Girouard and I.~Polterovich, \emph{Spectral geometry of the Steklov problem}, J. Spectr. Theory \textbf{7} (2017), no.~2, 321--359.

\bibitem{Hon21} Han~Hong, \emph{Higher dimensional surgery and Steklov eigenvalues}, J. Geom. Anal. \textbf{31} (2021), no.~12, 11931--11951.

\bibitem{K72} J.~R.~Kuttler, \emph{Remarks on a Stekloff eigenvalue problem}, SIAM J. Numer. Anal. \textbf{9} (1972), 1--5.

\bibitem{K79} J.~R.~Kuttler, \emph{Dirichlet eigenvalues},  SIAM J. Numer. Anal. \textbf{16} (1979), no.~2, 332--338.

\bibitem{KS68} J.~R.~Kuttler and V.~G.~Sigillito, \emph{Inequalities for membrane and Stekloff eigenvalues}, J. Math. Anal. Appl. \textbf{23} (1968), 148--160.

\bibitem{KKK14} N.~Kuznetsov, T.~Kulczycki, M.~Kwa\'{s}nicki, A.~Nazarov, S.~Poborchi, I.~Polterovich, and B.~Siudeja, \emph{The legacy of Vladimir Andreevich Steklov}, Notices Amer. Math. Soc. \textbf{61} (2014), no.~1, 9--22.

\bibitem{KW23} K.~K.~Kwong and Y.~Wei, \emph{Geometric inequalities involving three quantities in warped product manifolds}, Adv. Math. \textbf{430} (2023), Paper No.~109213.

\bibitem{LP22} P.~D.~Lamberti and L.~Provenzano, \emph{On the explicit representation of the trace space $H^{\frac{3}{2}}$ and of the solutions to biharmonic Dirichlet problems on Lipschitz domains via multi-parameter Steklov problems}, Rev. Mat. Complut. \textbf{35} (2022), no.~1, 53--88.

\bibitem{Las17} Monika~Laskawy, \emph{Optimality conditions of the first eigenvalue of a fourth order Steklov problem}, Commun. Pure Appl. Anal. \textbf{16} (2017), no.~5, 1843--1859.


\bibitem{Liu11} G.~Liu, \emph{The Weyl-type asymptotic formula for biharmonic Steklov eigenvalues on Riemannian manifolds}, Adv. Math. \textbf{228} (2011), no.~4, 2162--2217.

\bibitem{Liu16} G.~Liu, \emph{On asymptotic properties of biharmonic Steklov eigenvalues}, J. Differential Equations \textbf{261} (2016), no.~9, 4729--4757.

\bibitem{P70} L.~E.~Payne, \emph{Some isoperimetric inequalities for harmonic functions}, SIAM J. Math. Anal. \textbf{1} (1970), 354--359.

\bibitem{RS15} S.~Raulot and A.~Savo, \emph{Sharp bounds for the first eigenvalue of a fourth-order Steklov problem},  J. Geom. Anal. \textbf{25} (2015), no. 3, 1602--1619.

\bibitem{SW02} W.~Stekloff, \emph{Sur les probl$\grave{e}$mes fondamentaux de la physique math$\grave{e}$matique}, Ann. Sci. $\acute{E}$cole Norm. Sup. \textbf{19} (1902), 191--259.


\bibitem{WX09} Q.~L.~Wang and C.~Y.~Xia, \emph{Sharp bounds for the first non-zero Stekloff eigenvalues},  J. Funct. Anal. \textbf{257} (2009), no.~8, 2635--2644.

\bibitem{Wei54} Robert~Weinstock, \emph{Inequalities for a classical eigenvalue problem}, J. Rational Mech. Anal. \textbf{3} (1954), 745--753.


\bibitem{XW18}  Changyu~Xia and Qiaoling~Wang, \emph{Eigenvalues of the Wentzell-Laplace operator and of the fourth order Steklov problems},  J. Differential Equations \textbf{264} (2018), 6486--6506.


\bibitem{Xio21} Changwei~Xiong, \emph{Optimal estimates for Steklov eigenvalue gaps and ratios on warped product manifolds}, Int. Math. Res. Not. IMRN 2021, no.~22, 16938--16962.

\bibitem{Xio22} Changwei~Xiong, \emph{On the spectra of three Steklov eigenvalue problems on warped product manifolds}, J. Geom. Anal. \textbf{32} (2022), no.~5, Paper No. 153, 35 pp.















\end{thebibliography}

\end{document}